\documentclass[11pt]{amsart}

\usepackage{amsmath,amssymb}
\usepackage{graphicx}
\usepackage{tikz}
\usepackage{float}
\usetikzlibrary{calc}
\usepackage{comment}
\usepackage[T1]{fontenc}
\usepackage[utf8]{inputenc}
\usepackage{lmodern}
\usepackage[a4paper,margin=1in]{geometry}
\usepackage{amsmath,amssymb,amsthm,mathtools,mathrsfs}
\usepackage{enumitem}
\usepackage[nopatch=footnote]{microtype}
\usepackage[
  colorlinks=true,
  linkcolor=blue,        
  citecolor=green!60!black, 
  urlcolor=magenta,      
  filecolor=cyan         
]{hyperref}

 \theoremstyle{plain}
\newtheorem{theorem}{Theorem}[section]
\newtheorem{proposition}[theorem]{Proposition}
\newtheorem{lemma}[theorem]{Lemma}
\newtheorem{corollary}[theorem]{Corollary}
\theoremstyle{definition}

\theoremstyle{definition} 
\newtheorem{remark}[theorem]{Remark}
\numberwithin{equation}{section}

\newcommand{\R}{\mathbb R}
\newcommand{\Z}{\mathbb Z}
\newcommand{\T}{\mathbb T}
\newcommand{\E}{\mathbb E}
\newcommand{\1}{\mathbf 1}
\newcommand{\cE}{\mathcal E}

\newcommand{\cA}{\mathcal A}
\newcommand{\cG}{\mathcal G}
\DeclareMathOperator{\supp}{supp}
\DeclareMathOperator{\diam}{diam}
\DeclareMathOperator{\diag}{diag}

\DeclareMathOperator{\dir}{dir}

\newcommand{\norm}[1]{\lVert #1\rVert}
\setlist[enumerate]{leftmargin=2em,itemsep=2pt,topsep=4pt}
\allowdisplaybreaks[1]
\title{Strichartz estimates on two-dimensional irrational tori}

\author{Ming Wang}
\address{School of Mathematics and Statistics, HNP-LAMA,
Central South University, Changsha, Hunan 410083, P.R. China}
\email{m.wang@csu.edu.cn}

\author{Yuanjia Xiong}
\address{School of Mathematics and Statistics, HNP-LAMA,
Central South University, Changsha, Hunan 410083, P.R. China}
\email{xiongyuanjia@csu.edu.cn}

\begin{document}

\begin{abstract}
We prove an \(L^4\) Strichartz estimate for the dispersion \(n_1^2+\alpha n_2^2\) on \(\mathbb T^2\), for every \(\alpha>0\). The loss is controlled by the Shannon entropy of the normalized squared Fourier coefficients and the cardinality of their support \(S\). In particular, we obtain a norm loss of \((\log\#S)^{1/2}\) on fixed time intervals, with constants locally uniform in \(\alpha\). This provides an irrational-dispersion counterpart of the cardinality estimate of Herr and Kwak [Forum Math. Pi (2024)] and improves the \(N^\varepsilon\) loss of Bourgain and Demeter [Ann. of Math. (2015)] to \((\log N)^{1/2}\) on frequency boxes. The proof combines weighted incidence estimates, smooth localization, and an entropy recursion.
\end{abstract}

\maketitle

\setcounter{tocdepth}{2}
\tableofcontents

\section{Introduction}

\subsection{Main results}

Let $\T=\R/(2\pi\Z)$ be equipped with normalized Haar measure and,
for $\alpha>0$, define
\begin{equation}\label{eq:extension}
 E_\alpha f(t,x)
 :=\sum_{n\in\Z^2}f(n)e^{i(n\cdot x+tQ_\alpha(n))},
 \qquad Q_\alpha(n)=n_1^2+\alpha n_2^2.
\end{equation}
For a nonzero finitely supported $f:\Z^2\to\mathbb C$, define
\begin{equation}\label{eq:coefficient-entropy}
 \pi(n)=\frac{|f(n)|^2}{\norm{f}_{\ell^2}^2},
 \qquad
 \mathsf H(f)=\sum_{n\in\supp f}\pi(n)\log\frac1{\pi(n)},
 \qquad
 \norm{f}_{\ell^2}^2=\sum_{n\in\Z^2}|f(n)|^2.
\end{equation}
Here $\mathsf H(f)$ is the Shannon entropy of the normalized squared
coefficients, with $\mathsf H(0)=0$.
In cardinality bounds, we abbreviate $\log(2+\#S)$ by $\log\#S$.

\begin{theorem}\label{thm:main}
Let $I=[\alpha_-,\alpha_+]\Subset(0,\infty)$. There is $C_I<\infty$
such that, for all $\alpha\in I$, finite nonempty $S\subset\Z^2$,
and complex valued $f$ supported in $S$,
\begin{equation}\label{eq:main}
 \|E_\alpha f\|_{L^4([0,2\pi]\times\T^2)}^4
 \le C_I\Bigl(1+\mathsf H(f)\log\#S\Bigr)\|f\|_{\ell^2}^4.
\end{equation}
\end{theorem}

Concavity of the logarithm gives
\begin{equation}\label{eq:entropy-cardinality}
 0\le\mathsf H(f)\le\log \#(\supp f)\le\log\#S
 \qquad(f\ne0,\ \supp f\subset S).
\end{equation}
Hence the following bounds hold.

\begin{corollary}\label{cor:cardinality}
Under the assumptions of Theorem~\ref{thm:main},
\begin{equation}\label{eq:cardinality}
 \|E_\alpha f\|_{L^4([0,2\pi]\times\T^2)}^4
 \le C_I(\log\#S)^2\|f\|_{\ell^2}^4.
\end{equation}
In particular, if $S_N=[-N,N]^2\cap\Z^2$, $N\ge2$, and
$\supp f\subset S_N$, then
\begin{equation}\label{eq:intro-box-bound}
 \|E_\alpha f\|_{L^4([0,2\pi]\times\T^2)}^4
 \le C_I(\log N)^2\|f\|_{\ell^2}^4.
\end{equation}
\end{corollary}

In Schr\"odinger form, set
$\Delta_\alpha=\partial_{x_1}^2+\alpha\partial_{x_2}^2$.
By conjugation and Parseval's identity,
Corollary~\ref{cor:cardinality} yields the following estimate
for $\alpha\in I$ and $g$ with $\supp\widehat g\subset S$,
where $S\subset\Z^2$ is finite:\footnote{Here and below, the notation $A\lesssim B$ means $A\le CB$ for some
absolute constant $C>0$, while $A\lesssim_{c_1,c_2,\cdots} B$ allows
the constant $C$ to depend on $c_1,c_2,\cdots$.}
\begin{equation}\label{eq:strichartz-cardinality}
 \|e^{it\Delta_\alpha}g\|_{L^4([0,2\pi]\times\T^2)}
 \lesssim_I (\log\#S)^{1/2}\|g\|_{L^2(\T^2)}.
\end{equation}
In particular, if $P_{\le N}$ denotes the Fourier projection onto
$S_N$, then for $g\in L^2(\T^2)$ and $N\ge2$,
\begin{equation}\label{eq:strichartz-box}
 \|e^{it\Delta_\alpha}P_{\le N}g\|_{L^4([0,2\pi]\times\T^2)}
 \lesssim_I(\log N)^{1/2}\|g\|_{L^2(\T^2)}.
\end{equation}


Herr and Kwak~\cite{HerrKwak} proved the sharp estimate
\begin{equation}\label{eq:intro-herr-kwak}
 \|E_1f\|_{L^4([0,2\pi]\times\T^2)}
 \lesssim(\log\#S)^{1/4}\|f\|_{\ell^2}
\end{equation}
on the standard torus ($\alpha=1$) for $f$ supported on $S\subset\Z^2$.
We adapt their approach to irrational dispersion, obtaining a bound
in terms of coefficient entropy and support cardinality.
On frequency boxes, this improves the Bourgain--Demeter~\cite{BourgainDemeter}
loss from $N^\varepsilon$ to $(\log N)^{1/2}$.
The optimal loss for irrational dispersion remains open.

\begin{remark}[General positive diagonal symbols]\label{rem:general-symbol}
For $\alpha,\beta>0$, define
\[
 E_{\alpha,\beta}f(t,x)
 =\sum_{n\in\Z^2}f(n)e^{i(n\cdot x+t(\alpha n_1^2+\beta n_2^2))}.
\]
For every fixed $T>0$ and complex coefficients $f$ supported in
a finite nonempty $S\subset\Z^2$,
\begin{equation}\label{eq:general-symbol-entropy}
 \|E_{\alpha,\beta}f\|_{L^4([0,T]\times\T^2)}^4
 \le C_{K,T}\bigl(1+\mathsf H(f)\log\#S\bigr)\|f\|_{\ell^2}^4,
 \qquad (\alpha,\beta)\in K\Subset(0,\infty)^2.
\end{equation}
This follows from Theorem~\ref{thm:main} by the time rescaling
$E_{\alpha,\beta}f(t,x)=E_{\beta/\alpha}f(\alpha t,x)$
and a finite covering of $[0,\alpha T]$ by translated intervals of
length $2\pi$. 
\end{remark}

\subsection{Background and related developments}

Bourgain~\cite{Bourgain} introduced periodic Fourier restriction
to study NLS, proving the critical $N^\varepsilon$ and scale-invariant
supercritical bounds on the standard $\T^2$. Herr--Kwak's sharp
endpoint is recalled in \eqref{eq:intro-herr-kwak}.
Liang, Shen, and Wang~\cite{LiangShenWang} extended the sharp
$(\log N)^{1/4}$ norm loss to frequency boxes for positive integral
diagonal symbols. Clearing denominators and rescaling time covers
fixed positive rational dispersion ratios, but not irrational ones.

Bourgain~\cite{BourgainIrrational} began the irrational-torus theory
in dimension three. For two-dimensional irrational tori,
Guo, Oh, and Wang~\cite{GuoOhWang} obtained scale-invariant
estimates in the supercritical range $p>20/3$.
Killip--Vi\c{s}an~\cite{KillipVisan} subsequently proved such
estimates for all $p>4$ on every rectangular torus.

At the critical exponent $p=4$, the problem is to determine
the size of the frequency loss.
Catoire--Wang~\cite{CatoireWang}, Demirbas~\cite{Demirbas}, and
Demeter~\cite{DemeterIncidence} successively improved the endpoint
$L^4$ bounds. Bourgain--Demeter decoupling~\cite{BourgainDemeter}
gives, for every $\alpha>0$ and $\varepsilon>0$,
\begin{equation}\label{eq:intro-decoupling}
 \|E_\alpha f\|_{L^4([0,2\pi]\times\T^2)}
 \lesssim_{\alpha,\varepsilon}N^\varepsilon\|f\|_{\ell^2},
 \qquad \supp f\subset S_N.
\end{equation}
The general irrational $L^4$ norm losses thus progress as
\begin{equation}\label{eq:intro-irrational-history}
 \underset{\text{\cite{CatoireWang}}}{N^{1/6}}
 \ \longrightarrow\
 \underset{\text{\cite{Demirbas}}}{N^{131/832+\varepsilon}}
 \ \longrightarrow\
 \underset{\text{\cite{DemeterIncidence}}}{N^{1/8+\varepsilon}}
 \ \longrightarrow\
 \underset{\text{\cite{BourgainDemeter}}}{N^\varepsilon}
 \ \longrightarrow\
 \underbrace{(\log N)^{1/2}}_{\substack{\text{present work}\\
                                      \text{Corollary~\ref{cor:cardinality}}}}.
\end{equation}

These upper bounds are complemented by an obstruction to a
lossless endpoint estimate.
Kishimoto~\cite[Theorem~1.4]{Kishimoto} ruled out a lossless strong
$L^4$ estimate on every fixed positive-length interval for almost
every positive irrational $\alpha$.

A related question concerns uniform estimates below the critical
exponent. Bourgain's uniform subcritical conjecture on the standard
$\T^d$~\cite[(3.2)]{Bourgain} predicts that, for $g\in L^2(\T^d)$,
\begin{equation}\label{eq:intro-subcritical-conjecture}
 \|e^{it\Delta}P_{\le N}g\|_{L^p([0,2\pi]\times\T^d)}
 \lesssim_{d,p}\|g\|_{L^2(\T^d)},
 \qquad 2<p<\frac{2(d+2)}d.
\end{equation}
Here $\Delta=\sum_{j=1}^d\partial_{x_j}^2$ and $P_{\le N}$
projects onto $[-N,N]^d\cap\Z^d$; the constant is independent of $N$.

Recent progress in dimension one connects this conjecture with
sharp endpoint estimates in terms of support cardinality.
For the integer parabola, Deng, Fan, Guo, Guo, and
Luo~\cite{DengFanGuoGuoLuo} recently proved, for finite nonempty
$S\subset\Z$ and complex coefficients $a$, 
\begin{align}
 \|\sum_{n\in S}a(n)e^{i(nx+n^2t)}\|_{L^{6,\infty}([0,2\pi]\times\T)}
 &\lesssim\|a\|_{\ell^2},\label{eq:intro-one-dimensional-uniform}\\
 \|\sum_{n\in S}a(n)e^{i(nx+n^2t)}\|_{L^6([0,2\pi]\times\T)}
 &\lesssim(\log\#S)^{1/6}\|a\|_{\ell^2}.\label{eq:intro-one-dimensional-L6}
\end{align}
Their proof of \eqref{eq:intro-one-dimensional-uniform}
refines the $p$-adic bilinearization and high-low arguments
of Guo, Li, and Yung~\cite{GuoLiYung}.
The uniform weak $L^6$ bound
\eqref{eq:intro-one-dimensional-uniform}, interpolated with the
$L^2$ bound, settles the conjecture for $d=1$.
The strong endpoint bound \eqref{eq:intro-one-dimensional-L6}
is sharp.
Li and Yung~\cite{LiYung} gave another proof of
\eqref{eq:intro-one-dimensional-L6} by refining the high-low
method of Guth, Maldague, and Wang~\cite{GuthMaldagueWang}.
The full conjectured subcritical range remains open in higher
dimensions.

These questions also have applications to nonlinear equations
and control theory. Strichartz estimates underlie the Cauchy theory
and Sobolev norm growth for nonlinear Schr\"odinger equations on
tori~\cite{Bourgain,CatoireWang,Demirbas}. In one dimension, Skouloudis and Yu~\cite{SkouloudisYu}
combine short-time $L^6$ Strichartz estimates with the $I$-method
to prove small-mass global well-posedness of the periodic
quintic NLS in $H^s(\T)$, $s>0$.
On the standard $\T^2$, Herr--Kwak~\cite{HerrKwak}
combine short-time $L^4$ and multilinear estimates to prove
small-mass global well-posedness of the cubic NLS in
$H^s(\T^2)$, $s>0$.
In another breakthrough, their later
work~\cite{HerrKwakGlobal} uses a new inverse Strichartz estimate
to establish global well-posedness of the defocusing cubic NLS
in $H^s(\T^2)$, $s>0$, for arbitrary mass.

The conjectured lossless subcritical bounds provide a further
motivation from control theory. By the criterion of Burq and
Zhu~\cite[Theorems~1.8 and~1.11]{BurqZhu}, a lossless $L^p$ bound
with $p>2$ on the standard torus implies observability from
every positive-measure space-time set.
For real-part controls, G. Wang and M.
Wang~\cite[Theorem~1.2]{WangWang} characterize the null-controllable
subspace at time $2\pi$ for shifted free equations on admissible
measurable product sets.

\subsection{Overview of the proof}

The central difficulty is to control near resonances on a fixed time
interval. For an irrational coefficient, time integration does not
reduce the fourth moment to a count of exact resonances. After
spatial frequency conservation is imposed, the relevant interaction
is governed by
$B_\alpha(u,v)=u_1v_1+\alpha u_2v_2$.
A positive time weight reduces the problem to parallelograms with
$|B_\alpha(u,v)|<\eta$, where $\eta=\eta(I)>0$ is small. In the coordinates
$T_\alpha=\diag(1,\sqrt\alpha)$, their nonzero side vectors are nearly
orthogonal. 

Our proof combines weighted incidence estimates with smooth
localization. Two features are essential to the final bound:
the incidence summation depends on the number of supported points
rather than their diameter, and the localization recursion preserves
the coefficient of the recursive term exactly. The latter allows
entropy dissipation to control the total error over all generations.

For a nonnegative finitely supported $f$, write
$\mu(f)=\|f\|_{\ell^2}^2$ and let $K(s)=(1-|s|/\eta)_+$. The discrete
estimate underlying the proof concerns the weighted energy
\begin{equation}\label{eq:energy}
 \cE_\alpha(f)
 :=\sum_{x,u,v\in\Z^2}
 f(x)f(x+u)f(x+v)f(x+u+v)K(B_\alpha(u,v)).
\end{equation}


\begin{proposition}\label{prop:energy}
For $\alpha\in I$ and $f\ge0$ with finite support $S=\supp f$,
\begin{equation}\label{eq:energy-bound}
 \cE_\alpha(f)
 \le\mu(f)^2+C_I\mu(f)^2\mathsf H(f)\log\#S.
\end{equation}
\end{proposition}

A positive time weight converts Proposition~\ref{prop:energy} into
Theorem~\ref{thm:main}. To prove the proposition, a weighted incidence
argument first bounds parallelograms in each aspect-ratio range.
Direction graphs reduce the estimate to matched pairs of directions,
where fixing a diagonal and one side direction uniquely completes
the parallelogram. Summation over actual pairs of distinct points
prevents the appearance of a logarithm depending on the diameter of
the support.

Thin parallelograms are then localized near the two endpoints of
their long side. The total mass of separated localization centers is
exactly $\kappa_LJ(f)$, where $\kappa_L$ is a normalization constant,
and the same $\kappa_L$ occurs in the coverage estimate. A
positive-kernel Cauchy--Schwarz estimate and an angular exclusion
therefore give a recursion with coefficient exactly one. A larger
coefficient would be amplified at successive generations of the
iteration. Here $J(f)=\sum_{x\ne y}f(x)^2f(y)^2$.
Writing $G(f)=\cE_\alpha(f)-\mu(f)^2$ and
$\mathscr H(f)=\mu(f)^2\mathsf H(f)$, the positive localization
operator $\mathscr L$, constructed in Subsection~\ref{sec:recursion},
satisfies
\[
 G(f)\le(\mathscr LG)(f)+C_I\log\#S\,J(f),
 \qquad
 \mathscr H(f)-(\mathscr L\mathscr H)(f)\ge(\log2)J(f).
\]
Iteration of the first estimate and telescoping in the second give
$G(f)\le C_I\log\#S\,\mathscr H(f)$. The iteration terminates
because every effective local support is a proper subset of its
parent support. Thus incidence counting supplies one logarithmic
factor, while entropy controls the accumulated localization error.

\subsection{Acknowledgments}
This work was partially supported by the National Natural Science Foundation of China (No.~12571260), the Natural Science Foundation of Hunan Province
(No.~2026JJ20012) and the Hunan Basic Science Research Center for Mathematical Analysis (2024JC2002).

During the preparation of this work, the authors used ChatGPT as an interactive discussion tool, with continuous guidance from the authors throughout the research process. The content was thoroughly rewritten, reviewed, and verified by the authors, who take full responsibility for the final manuscript.

\section{Proof of the main theorem}
\label{sec:extension}

We first prove Theorem~\ref{thm:main} assuming
Proposition~\ref{prop:energy}. We then introduce the localization
operator and state the key estimates used to prove that proposition.

Fix $I=[\alpha_-,\alpha_+]\Subset(0,\infty)$ and put
\begin{equation}\label{eq:lattice-parameters}
 T_\alpha=\diag(1,\sqrt\alpha),\qquad
 B_\alpha(u,v)=T_\alpha u\cdot T_\alpha v
             =u_1v_1+\alpha u_2v_2,\qquad
 \lambda=\min\{1,\sqrt{\alpha_-}\}.
\end{equation}
Here $\diag(1,\sqrt\alpha)$ is the diagonal matrix with the stated
entries. We call the plane containing $T_\alpha\Z^2$ the
\emph{transformed plane}. We use Euclidean lengths and angles in the transformed plane, where $Q_\alpha(u)=|T_\alpha u|^2$.
Moreover, $|T_\alpha u|\ge\lambda|u|$, so every nonzero transformed
lattice vector has length at least $\lambda$.

Choose once and for all, as a function of $I$,
\begin{equation}\label{eq:eta}
 \eta=10^{-6}\min\{1,\alpha_-,\sqrt{\alpha_-}\},
 \qquad K(s)=\left(1-\frac{|s|}{\eta}\right)_+,
\end{equation}
where $t_+=\max\{t,0\}$.
Thus $K$ is even, $0\le K\le1$, $K(0)=1$, and $K(s)=0$ for
$|s|\ge\eta$. In particular, a term with positive weight
$K(B_\alpha(u,v))$ has $|B_\alpha(u,v)|<\eta$.
For $u,v\ne0$ this means that the transformed side vectors are
nearly orthogonal, in the precise sense
\[
 \frac{|T_\alpha u\cdot T_\alpha v|}
 {|T_\alpha u|\,|T_\alpha v|}
 <\frac{\eta}{|T_\alpha u|\,|T_\alpha v|}\leq 10^{-6}.
\]

For a finitely supported function $f:\Z^2\to[0,\infty)$, define
\begin{align}\label{eq:mass-functionals}
 \mu(f)&=\sum_x f(x)^2=\norm{f}_{\ell^2}^2,
 &\nu(f)&=\sum_x f(x)^4,\notag\\
 J(f)&=\mu(f)^2-\nu(f)
      =\sum_{x\ne y}f(x)^2f(y)^2.
\end{align}

The quantity $J(f)$ is the mass of the ordered pairs of distinct
points, with each point given weight $f(x)^2$. The energy $\cE_\alpha$ defined in \eqref{eq:energy} is a
nonnegative quartic counting functional, it satisfies 
$\cE_\alpha(cf)=c^4\cE_\alpha(f)$ for $c\ge0$.

Let $\cE_\alpha^*(f)$ denote the sum in \eqref{eq:energy} restricted
to $u,v\ne0$. Then we have the decomposition
\begin{equation}\label{eq:zero-edges}
 \cE_\alpha(f)=\cE_\alpha^*(f)+2\mu(f)^2-\nu(f).
\end{equation}
Indeed, the case $u=0$ and $v=0$ each contribute $\mu(f)^2$, with overlap
$\nu(f)$.

If $u,v\ne0$ are parallel after applying $T_\alpha$, then
$|B_\alpha(u,v)|=|T_\alpha u|\,|T_\alpha v|\ge\lambda^2>\eta$.
Consequently every positive-weight term of $\cE_\alpha^*(f)$ is a
nondegenerate parallelogram.

For the rest of the paper, $I$, $\lambda$, $\eta$, and $K$ are fixed
as above. When $\alpha$ is fixed, we abbreviate $T_\alpha$ and
$B_\alpha$ by $T$ and $B$. In the purely discrete sections,
$x,y,u,v$ denote lattice points or lattice vectors; the transformed
points are $Tx,Ty,Tu,Tv$.

\subsection{Reduction of Theorem~\ref{thm:main} to Proposition~\ref{prop:energy}}
\begin{lemma}\label{lem:time-window}
Let
\begin{equation}\label{eq:time-weight}
 W(t)=\frac{\eta}{\pi}
       \left(\frac{\sin(\eta t)}{\eta t}\right)^2,
 \qquad W(0)=\frac{\eta}{\pi}.
\end{equation}
Then $W\ge0$, $\int_\R W(t)\,dt=1$, and
\begin{equation}\label{eq:Fourier-window}
 \int_\R W(t)e^{2its}\,dt=K(s).
\end{equation}
Moreover, $W(t)\ge c_I>0$ for $0\le t\le2\pi$.
For every finitely supported complex valued $f$,
\begin{align}\label{eq:moment-identity}
 &\int_\R\int_{\T^2}W(t)|E_\alpha f(t,x)|^4\,dx\,dt\notag\\
 &\quad=
 \sum_{n,u,v\in\Z^2}
 f(n)\overline{f(n+u)}\,\overline{f(n+v)}f(n+u+v)
 K(B_\alpha(u,v)).
\end{align}
\end{lemma}

\begin{proof}
Clearly, $W\ge0$ and $W\in L^1(\R)$. With the convention
$\widehat W(\xi)=\int_\R W(t)e^{-it\xi}\,dt$, a direct calculation gives
\[
 \frac{1}{2\pi}\int_{-2\eta}^{2\eta}
 \left(1-\frac{|\xi|}{2\eta}\right)e^{it\xi}\,d\xi
 =\frac{\sin^2(\eta t)}{\pi\eta t^2}=W(t),\qquad t\ne0,
\]
with the value at $t=0$ understood by continuity. Fourier inversion
therefore yields
\[
 \widehat W(\xi)=\left(1-\frac{|\xi|}{2\eta}\right)_+=K(\xi/2).
\]
Taking \(\xi=-2s\) yields \eqref{eq:Fourier-window}, while setting \(\xi=0\) gives
\(\int_{\mathbb R} W(t)\,dt=1.\)
 For \(0\le t\le 2\pi\), we have
\[
\frac{\sin(\eta t)}{\eta t}\ge \frac{2}{\pi},
\]
and hence
one may take \(c_I:=4\eta/\pi^3\).
Expanding \(Q_\alpha\), we obtain
\begin{equation}\label{eq:quadratic-identity}
Q_\alpha(n)+Q_\alpha(n+u+v)
-Q_\alpha(n+u)-Q_\alpha(n+v)
=
2B_\alpha(u,v).
\end{equation}
It follows from \eqref{eq:Fourier-window} that
\[
\begin{aligned}
\int_{\mathbb R}\int_{\mathbb T^2}
W(t)\lvert E_\alpha f(t,x)\rvert^4
\,dx\,dt
&=
\sum_{n,u,v\in\mathbb Z^2}
f(n)\overline{f(n+u)}\,
\overline{f(n+v)}f(n+u+v)
\\
&\qquad\times
\int_{\mathbb R}
W(t)e^{2itB_\alpha(u,v)}\,dt
\\
&=
\sum_{n,u,v\in\mathbb Z^2}
f(n)\overline{f(n+u)}\,
\overline{f(n+v)}f(n+u+v)
K\bigl(B_\alpha(u,v)\bigr).
\end{aligned}
\]
This proves \eqref{eq:moment-identity}. 
\end{proof}

Assuming Proposition~\ref{prop:energy}, the preceding lemma yields
Theorem~\ref{thm:main} as follows.

\begin{proof}[Proof of Theorem~\ref{thm:main}]
By Lemma~\ref{lem:time-window}, the nonnegativity of \(K\) and
Proposition~\ref{prop:energy},
we obtain
\begin{align*}
\|E_\alpha f\|_{L_{t,x}^4([0,2\pi]\times\mathbb T^2)}^4
&\le
c_I^{-1}
\int_{\mathbb R}\int_{\mathbb T^2}
W(t)\lvert E_\alpha f(t,x)\rvert^4\,dx\,dt
\\
&\le
c_I^{-1}\mathcal E_\alpha(|f|)
\\
&\le
C_I
\Bigl(
1+\mathsf H(f)\log\#S
\Bigr)
\|f\|_{\ell^2}^4.
\end{align*}
This completes the proof of Theorem~\ref{thm:main}.
\end{proof}

\subsection{The localization operator and key estimates}
\label{sec:recursion}

In the transformed coordinates, a thin parallelogram consists
of two widely separated pairs of nearby vertices. We localize
each pair in a smooth window whose radius is large compared
with the short side and small compared with the long side.
The aim is to control the interaction between the two localized
pieces by their individual energies, and then recover the thin
contribution by averaging over the window centers and scales.

Fix a nonnegative even function $\phi\in C_c^\infty(\R^2)$,
supported in the unit ball and normalized by
$\int_{\R^2}\phi(Z)^2\,dZ=1$. Fix
\begin{equation}\label{eq:localization-parameters}
 L\ge128,\qquad \tau\ge100L.
\end{equation}
These choices remain fixed throughout the proof.

For a separation scale $R>0$, put $r=R/L$ and define
\begin{equation}\label{eq:soft-localization}
 \Phi_{R,z}(x)=\phi\left(\frac{Tx-z}{r}\right),\qquad
 d\gamma_r(z)=\frac{dz}{r^2},\qquad
 f_{R,z}(x)=f(x)\Phi_{R,z}(x).
\end{equation}
Write $\mu=\mu(f), \nu=\nu(f)$. At a fixed scale, we abbreviate
$\Phi_{R,z},f_{R,z}$ by $\Phi_z,f_z$ and put
\begin{align}\label{eq:muz-fz}
  \mu_z=\mu(f_z),\qquad \nu_z=\nu(f_z).   
\end{align}
The normalization of $\phi$ gives 
\begin{equation}\label{eq:partition-mass}
\int_{\R^2}\mu_z\,d\gamma_r(z)=\mu,
 \qquad \diam(T\supp f_z)\le2r.
\end{equation}
Define the annulus of separated center pairs by
\[
 \cA_R=\{(z,w)\in\R^2\times\R^2:R\le|w-z|<2R\}
\]
and its weighted mass by
\begin{equation}\label{eq:center-mass}
 \mathfrak m_R(f)=\iint_{\cA_R}\mu_z\mu_w\,d\gamma_r(z)d\gamma_r(w).
\end{equation}

The following identity provides the exact normalization of the
localization operator. Its proof is deferred to
Subsection~\ref{subsec:mass-entropy}.

\begin{lemma}\label{lem:exact-mass}
There exists a constant $0<\kappa_L<\infty$ depending only
on $\phi$ and $L$, such that for all $x,y\in\Z^2$,
\begin{equation}\label{eq:pairwise-exact-mass}
 \int_0^\infty\iint_{\cA_R}
       \Phi_z(x)^2\Phi_w(y)^2\,
       d\gamma_r(z)d\gamma_r(w)\frac{dR}{R}
 =\kappa_L\1_{\{x\ne y\}}.
\end{equation}
In particular,
\begin{equation}\label{eq:exact-mass}
 \int_0^\infty\mathfrak m_R(f)\,\frac{dR}{R}
 =\kappa_LJ(f)=\kappa_L(\mu^2-\nu).
\end{equation}
\end{lemma}

We now organize the transfer to localized energies as a positive
operator. Its mass normalization preserves the coefficient of the
recursive term, while entropy bounds the total error accumulated
under iteration.

Define the positive localization operator by
\begin{equation}\label{eq:entropy-localization}
 (\mathscr L\Psi)(f)=\frac1{\kappa_L}
 \int_0^\infty\iint_{\cA_R}
 \frac{\mu_w}{\mu_z}\Psi(f_z)\,
 d\gamma_r(z)d\gamma_r(w)\frac{dR}{R}
\end{equation}
for nonnegative quartic homogeneous functionals $\Psi$, with the
integrand set to zero when $\mu_z\mu_w=0$. The operator is linear in $\Psi$ and preserves pointwise inequalities.
Further iterations localize each resulting function again,
with the mass ratios computed from its own localizations.
We set $\mathscr L^0\Psi=\Psi$.

Put $\mathcal P(f)=\mu(f)^2$.
The choice of weights gives
\begin{equation}\label{eq:operator-mass}
 (\mathscr L\mathcal P)(f)
=\frac1{\kappa_L}\int_0^\infty\iint_{\cA_R}\mu_z\mu_w\,d\gamma_r(z)d\gamma_r(w)\frac{dR}{R}
 =J(f) 
\end{equation}
by Lemma~\ref{lem:exact-mass}. In particular,
\begin{equation}\label{eq:operator-pair-mass}
 0\le\mathscr LJ\le\mathscr L\mathcal P=J. 
\end{equation}

For functions $g$  supported in a fixed finite set $S$, the bound
$|\Psi(g)|\le C_S\mu(g)^2$ implies
$|\mathscr L\Psi(g)|\le C_SJ(g)$.
Hence $\mathscr L$ also acts on signed functionals with this bound. For example, $\nu(g),J(g)\le\mu(g)^2$ and 
\begin{align}\label{eq:crude-energy-finiteness}
 \cE_\alpha(g)
  \le\sum_{n\in\Z^2}
    \left(\sum_xg(x)g(n-x)\right)^2 \le\#S\sum_{n,x}g(x)^2g(n-x)^2
  =\#S\,\mu(g)^2.
\end{align}

Define the nonnegative energy excess
\[
 G(f)=\cE_\alpha(f)-\mathcal P(f)
     =\cE_\alpha^*(f)+J(f)\ge0,
\]
where the last identity is \eqref{eq:zero-edges}.

\begin{lemma}\label{lem:energy-recursion}
For a nonnegative function $f$ supported in a finite nonempty set
$S\subset\Z^2$,
\begin{equation}\label{eq:entropy-energy-recursion}
 G(f)\le(\mathscr LG)(f)
             +C_I\log\#SJ(f).
\end{equation}
\end{lemma}

The proof of Lemma~\ref{lem:energy-recursion} combines the
aspect-ratio bounds with the localization estimates; it is given in
Subsection~\ref{subsec:energy-recursion-proof}.

The error in Lemma~\ref{lem:energy-recursion} is proportional to
$J$. To sum these errors, we seek a nonnegative functional whose
decrease under $\mathscr L$ controls $J$.

For a probability distribution $p$ on a finite set, define its
entropy by $\operatorname{Ent}(p)=\sum_xp(x)\log(1/p(x))$,
with the convention $0\log(1/0)=0$, and set
\begin{equation}\label{eq:entropy-functional}
 \mathscr H(f)=\mu(f)^2\mathsf H(f),\qquad \mathscr H(0)=0.
\end{equation}
The factor $\mu(f)^2$ gives the same homogeneity as the energy.

The functional $\mathscr H$ has the required property, as
shown by the following lemma.

\begin{lemma}\label{lem:entropy-dissipation}
For every nonnegative finitely supported $f$,
\begin{equation}\label{eq:entropy-dissipation}
 \mathscr H(f)-(\mathscr L\mathscr H)(f)\ge(\log2)J(f).
\end{equation}
Consequently, for every integer $n\ge1$,
\begin{equation}\label{eq:entropy-telescoping}
 \sum_{j=0}^{n-1}(\mathscr L^jJ)(f)
\le\frac{\mathscr H(f)}{\log2}.
\end{equation}
\end{lemma}

Lemma~\ref{lem:entropy-dissipation} is proved in
Subsection~\ref{subsec:mass-entropy}. We now show how
the two estimates complete Proposition~\ref{prop:energy}.

\subsection{Proof of Proposition~\ref{prop:energy}}

For a nonnegative finitely supported function $g$, write
\[
g_{R,z}(x)=g(x)\Phi_{R,z}(x),
\qquad
\mu_{R,z}(g)=\mu(g_{R,z})
\]
 as in \eqref{eq:soft-localization} and define
\[
\cA_R^+(g)
=\{(z,w)\in\cA_R:\mu_{R,z}(g)\mu_{R,w}(g)>0\}.
\]
An \emph{effective child} of $g$ is a localization $g_{R,z}$
for which $(z,w)\in\cA_R^+(g)$ for some $R>0$ and $w$.
Define the successive generations recursively by
\[
\mathcal D_0(g)=\{g\},
\qquad
\mathcal D_{j+1}(g)
=\{h_{R,z}:h\in\mathcal D_j(g),\ R>0,\ (z,w)\in\cA_R^+(h)\}.
\]
Thus each child is obtained by localizing its current parent,
and the positivity condition uses the local masses of that
parent. The elements of $\mathcal D_j(g)$ are the descendants
of $g$ in generation $j$.

If $g_{R,z}$ is an effective child, its paired localization
$g_{R,w}$ has nonempty support disjoint from that of $g_{R,z}$,
since $|w-z|\ge R>2R/L$. Both supports lie in $\supp g$, so
\begin{equation}\label{eq:support-contraction}
	\#(\supp g_{R,z})\le\#(\supp g)-1.
\end{equation}
In particular, the support cardinality strictly decreases
along every sequence of effective children.

The iterations of the operator follow the same recursive rule.
For every nonnegative quartic homogeneous functional $\Psi$
and every integer $n\ge0$, \eqref{eq:entropy-localization} gives
\begin{equation}\label{eq:iterated-localization}
	(\mathscr L^{n+1}\Psi)(g)
	=\frac1{\kappa_L}\int_0^\infty\iint_{\cA_R^+(g)}
	\frac{\mu_{R,w}(g)}{\mu_{R,z}(g)}
	(\mathscr L^n\Psi)(g_{R,z})
	\,d\gamma_{R/L}(z)d\gamma_{R/L}(w)\frac{dR}{R}.
\end{equation}
Here $(\mathscr L^n\Psi)(g_{R,z})$ represents $n$ further
generations of localization starting from the effective child
$g_{R,z}$. Thus $\Psi$ is evaluated on descendants
of $g_{R,z}$ in generation $n$, which are descendants of $g$
in generation $n+1$. At each subsequent generation, the
localizations and the mass ratios are computed from the
current parent.

\begin{proof}[Proof of Proposition~\ref{prop:energy}]
Fix a nonzero nonnegative $f$ and write $S=\supp f$. Lemma~\ref{lem:energy-recursion} gives
\begin{align}\label{eq:proof-prop-1}
G(g)\le(\mathscr LG)(g)+C_I\log\#S J(g)
\end{align}
for every nonnegative function $g$ supported in $S$.
All effective children of $g$ are again supported in $S$. Thus positivity of $\mathscr L$ allows us to apply
$\mathscr L^j$ to \eqref{eq:proof-prop-1}, with the same factor
$\log\#S$, to obtain
\begin{align}\label{eq:proof-prop-2}
(\mathscr L^jG)(f)
\le(\mathscr L^{j+1}G)(f)
+C_I\log\#S(\mathscr L^jJ)(f)
\qquad(j\ge0).
\end{align}
Summing \eqref{eq:proof-prop-2} over $j=0,\ldots,n-1$ and cancelling the intermediate
terms gives
\begin{equation}\label{eq:iterated-energy-recursion}
	G(f)\le(\mathscr L^nG)(f)
	+C_I\log\#S\sum_{j=0}^{n-1}(\mathscr L^jJ)(f)
	\qquad(n\ge1).
\end{equation}

We prove by induction that, for every $n\ge0$,
\begin{align}\label{eq:proof-prop-3}
\#(\supp g)\le n+1
\quad\Longrightarrow\quad
(\mathscr L^nG)(g)=0.
\end{align}
For $n=0$, any nonnegative function $g$ supported on at most
one point satisfies $\cE_\alpha^*(g)=J(g)=G(g)=0$.
Suppose the claim holds for $n-1$, where $n\ge1$, and let
$\#(\supp g)\le n+1$. By \eqref{eq:support-contraction},
every effective child $g_{R,z}$ has support cardinality
at most $n$. The induction hypothesis therefore gives
\[
(\mathscr L^{n-1}G)(g_{R,z})=0.
\]
Substitution into \eqref{eq:iterated-localization} proves
$(\mathscr L^nG)(g)=0$, completing the induction.

In particular, \eqref{eq:proof-prop-3} implies that $(\mathscr L^{\#S}G)(f)=0$. Taking $n=\#S$ in \eqref{eq:iterated-energy-recursion} and
applying \eqref{eq:entropy-telescoping} from
Lemma~\ref{lem:entropy-dissipation}, we conclude that
\[
G(f)\le\frac{C_I}{\log2}\log(\#S)\mathscr H(f)
=\frac{C_I}{\log2}\mu(f)^2\mathsf H(f)\log(\#S).
\]
Absorbing $1/\log2$ into $C_I$ and adding $\mathcal P(f)$ proves
\eqref{eq:energy-bound}.
\end{proof}

\section{Weighted incidence and aspect-ratio estimates}
\label{sec:incidences}

The goal of this section is Proposition~\ref{prop:aspect} and its
consequences in Corollary~\ref{cor:aspect-sums}. These bounds
control the contribution of moderate aspect ratios and the coverage
error for thin parallelograms in Lemma~\ref{lem:energy-recursion}.
We first prove the weighted direction-graph estimate and then sum
over the actual long point pairs.

\subsection{Direction graphs and positive kernels}

For $w\in\R^2\setminus\{0\}$, its unoriented direction is
\begin{equation}\label{eq:direction-definition}
 \dir(w)=\R w=\{tw:t\in\R\}.
\end{equation}
Thus $\dir(w)=\dir(cw)$ for every nonzero real $c$.
The direction space is denoted by $\mathbb{RP}^1$.
We write $\ell\parallel\theta$ if the affine line $\ell$ is a
translate of $\theta$, and $w\parallel\theta$ if $w\in\theta$;
the latter includes $w=0$.

A \emph{simple graph on unoriented directions} is a pair
\[
\mathcal G=(V(\mathcal G),E(\mathcal G)),
\]
where
\[
V(\mathcal G)\subset \mathbb{RP}^1,
\qquad
E(\mathcal G)\subset
\bigl\{
\{\theta,\theta'\}:
\theta,\theta'\in V(\mathcal G),\ \theta\neq\theta'
\bigr\}.
\]
The elements of \(V(\mathcal G)\) and \(E(\mathcal G)\) are called the
vertices and edges of \(\mathcal G\), respectively. Thus, ``simple''
means that \(\mathcal G\) has no loops or multiple edges. 
We write $\theta\sim_\cG\theta'$ when
$\{\theta,\theta'\}\in E(\cG)$.

For \(\theta\in\mathbb{RP}^1\), define
\begin{equation}\label{eq:graph-degree}
N_{\mathcal G}(\theta)
=
\{\theta'\in V(\mathcal G):\theta\sim_{\mathcal G}\theta'\},
\qquad
\deg_{\mathcal G}(\theta)
=
\#N_{\mathcal G}(\theta),
\end{equation}
where directions outside \(V(\mathcal G)\) are assigned degree zero.
We say that \(\mathcal G\) has maximum degree at most \(d\) if
\begin{equation}\label{eq:max-degree}
\sup_{\theta\in\mathbb{RP}^1}
\deg_{\mathcal G}(\theta)
\le d.
\end{equation}

Throughout, we assume \(d\ge1\), even if \(\mathcal G\) has no edges and allow infinite vertex sets.
For nonnegative finitely supported functions $f_1,f_2,f_3,f_4$ on
$\Z^2$, define
\begin{align}\label{eq:graph-form}
 \mathcal T_\cG(f_1,f_2,f_3,f_4)
 :=\sum_{\substack{x,u,v\in\Z^2\\u,v\ne0\\
       \{\dir(Tu),\dir(Tv)\}\in E(\cG)}}
 f_1(x)f_2(x+u)f_3(x+v)f_4(x+u+v)K(B(u,v)).
\end{align}

For example, let \(\mathcal G\) consist of the single edge joining
\(\mathbb R(1,0)\) and \(\mathbb R(0,1)\). Then \(\mathcal G\) has
maximum degree one, and \(\mathcal T_{\mathcal G}\) gives a weighted
count of parallelograms with one pair of horizontal sides and one pair
of vertical sides.

The associated graph energy is defined by
\begin{equation}\label{eq:graph-energy}
\mathcal E_{\mathcal G}(f)
:=
\mathcal T_{\mathcal G}(f,f,f,f).
\end{equation}

A kernel $L:\Z^2\to\mathbb C$ is called \emph{positive definite}
here if, for every finitely supported complex function $c$,
\[
 \sum_{x,y\in\Z^2}\overline{c(x)}c(y)L(y-x)\ge0.
\]
Such a kernel induces the positive semidefinite sesquilinear form
\[
\langle F,G\rangle_L
:=
\sum_{x,y\in\mathbb Z^2}
\overline{F(x)}\,G(y)L(y-x),
\]
which satisfies the Cauchy--Schwarz inequality
\[
|\langle F,G\rangle_L|^2
\le
\langle F,F\rangle_L
\langle G,G\rangle_L.
\]

For every fixed $u$, the kernel $v\mapsto K(B(u,v))$ is positive
definite. Indeed, \eqref{eq:Fourier-window} gives the explicit
representation
\begin{equation}\label{eq:positive-kernel-representation}
 \sum_{x,y}\overline{c(x)}c(y)K(B(u,y-x))
 =\int_\R W(t)
       \left|\sum_x c(x)e^{2itB(u,x)}\right|^2\,dt\ge0.
\end{equation}

\begin{lemma}
\label{lem:graph-holder}
Suppose that \(\mathcal G\) has maximum degree at most \(d\ge1\).
For a nonnegative finitely supported function
\(f:\mathbb Z^2\to[0,\infty)\), define
\[
H_{\mathcal G}(f)
:=
\mathcal E_{\mathcal G}(f)
+
d\|f\|_{\ell^2}^4.
\]
Then, for any nonnegative finitely supported functions
\(f_1,f_2,f_3,f_4:\mathbb Z^2\to[0,\infty)\),
\begin{equation}\label{eq:graph-holder}
\mathcal T_{\mathcal G}(f_1,f_2,f_3,f_4)
\le
2\prod_{i=1}^4 H_{\mathcal G}(f_i)^{1/4}.
\end{equation}
\end{lemma}


\begin{proof}
For \(u\neq0\), define
\[
L_u(h)
=
K(B(u,h))
\sum_{\theta\in N_{\mathcal G}(\operatorname{dir}(Tu))}
\mathbf 1_{\{Th\parallel\theta\}}.
\]
For fixed \(\theta\in N_{\mathcal G}(\operatorname{dir}(Tu))\),
since \(T(y-x)\parallel\theta\) if and only if \(Tx\) and \(Ty\)
lie on the same affine line \(\ell\parallel\theta\), we have
\[
\begin{aligned}
\sum_{x,y\in\mathbb Z^2}
\overline{c(x)}c(y)
K(B(u,y-x))
\mathbf 1_{\{T(y-x)\parallel\theta\}}
&=
\int_{\mathbb R}
W(t)
\sum_{\ell\parallel\theta}
\Big|
\sum_{x\in\mathbb Z^2,Tx\in\ell}
c(x)e^{2itB(u,x)}
\Big|^2
\,dt
\ge0,
\end{aligned}
\]
where only finitely many lines \(\ell\) contribute, since \(c\) is finitely supported.
Summing over
\(\theta\in N_{\mathcal G}(\operatorname{dir}(Tu))\)
shows that \(L_u\) is positive definite. Moreover,
\(L_u(h)\ge0\) for all \(h\in\mathbb Z^2\).

Set
\[
F_u^{f,g}(x)=f(x)g(x+u),
\qquad
A(f,g)
=
\sum_{u\neq0}
\langle F_u^{f,g},F_u^{f,g}\rangle_{L_u}.
\]
By a straightforward computation,
\[
\begin{aligned}
A(f,g)
&=
\sum_{u\neq0}
\sum_{x,y\in\mathbb Z^2}
f(x)g(x+u)f(y)g(y+u)L_u(y-x)
\\
&=
\sum_{\substack{u\neq0\\x,v\in\mathbb Z^2}}
f(x)g(x+u)f(x+v)g(x+u+v)
K(B(u,v))
\sum_{\theta\in N_{\mathcal G}(\operatorname{dir}(Tu))}
\mathbf 1_{\{Tv\parallel\theta\}}.
\end{aligned}
\]
Separating the cases \(v\neq0\) and \(v=0\), we obtain
\begin{align}\label{eq:graph-collision}
A(f,g)
&=
\mathcal T_{\mathcal G}(f,g,f,g)
+
\sum_{u\neq0}
\deg_{\mathcal G}(\operatorname{dir}(Tu))
\sum_{x\in\mathbb Z^2}
f(x)^2g(x+u)^2
\notag\\
&\le
\mathcal T_{\mathcal G}(f,g,f,g)
+
d\|f\|_{\ell^2}^2\|g\|_{\ell^2}^2.
\end{align}
In particular, \(A(f,f)\le H_{\mathcal G}(f).\)

Since the \(v=0\) terms are nonnegative, we may add them to the sum and obtain
\begin{align}\label{eq:graph-first-cauchy}
\mathcal T_{\mathcal G}(f_1,f_2,f_3,f_4)
&\le
\sum_{u\neq0}
\left\langle
F_u^{f_1,f_2},
F_u^{f_3,f_4}
\right\rangle_{L_u}
\notag\\
&\le
\sum_{u\neq0}
\left\langle
F_u^{f_1,f_2},
F_u^{f_1,f_2}
\right\rangle_{L_u}^{1/2}
\left\langle
F_u^{f_3,f_4},
F_u^{f_3,f_4}
\right\rangle_{L_u}^{1/2}
\notag\\
&\le
\left(
\sum_{u\neq0}
\left\langle
F_u^{f_1,f_2},
F_u^{f_1,f_2}
\right\rangle_{L_u}
\right)^{1/2}
\left(
\sum_{u\neq0}
\left\langle
F_u^{f_3,f_4},
F_u^{f_3,f_4}
\right\rangle_{L_u}
\right)^{1/2}
\notag\\
&=
A(f_1,f_2)^{1/2}
A(f_3,f_4)^{1/2},
\end{align}
the last two inequalities follow from two applications of the
Cauchy--Schwarz inequality. By symmetry in \(u\) and \(v\) and
\eqref{eq:graph-first-cauchy},
\[
\mathcal T_{\mathcal G}(f,g,f,g)
\le
A(f,f)^{1/2}A(g,g)^{1/2}
\le
H_{\mathcal G}(f)^{1/2}H_{\mathcal G}(g)^{1/2}.
\]
Furthermore, \(d\|f\|_{\ell^2}^2\|g\|_{\ell^2}^2
\le
H_{\mathcal G}(f)^{1/2}H_{\mathcal G}(g)^{1/2}.\)

Combining the preceding two estimates with
\eqref{eq:graph-collision} gives \(\)
\[
A(f,g)
\le
2H_{\mathcal G}(f)^{1/2}H_{\mathcal G}(g)^{1/2}.
\]
Substituting this into \eqref{eq:graph-first-cauchy} proves
\eqref{eq:graph-holder}.
\end{proof}

\subsection{Weighted parallelograms with matched directions}
We recall some results from incidence geometry.
If $P\subset\R^2$ is finite and $\mathcal L$ is a finite family
of distinct lines, a \emph{point--line incidence} is a pair
$(x,\ell)\in P\times\mathcal L$ with $x\in\ell$.
Its total number is
\[
 \mathcal I(P,\mathcal L)
 =\#\{(x,\ell)\in P\times\mathcal L:x\in\ell\}
 =\sum_{\ell\in\mathcal L}\#(P\cap\ell).
\]
The Szemer\'edi--Trotter theorem \cite{SzemerediTrotter} states that
\begin{equation}\label{eq:ST}
 \mathcal I(P,\mathcal L)
 \le C\bigl((\#P)^{2/3}(\#\mathcal L)^{2/3}
                  +\#P+\#\mathcal L\bigr).
\end{equation}
It implies the following rich-line bound.
\begin{lemma}[\cite{TaoVu}, Corollary~8.5]
\label{lem:rich-lines}
Let \(P\subset\mathbb R^2\) be a set of \(n\) points, where
\(n\in\mathbb N\), and let \(k\ge2\) be an integer. The number \(m\)
of lines in \(\mathbb R^2\) passing through at least \(k\) points of
\(P\) satisfies
\begin{equation}\label{eq:rich-lines}
m\lesssim \frac{n^2}{k^3}+\frac{n}{k}.
\end{equation}
\end{lemma}

A simple graph \(\mathcal M\) is called a \emph{matching} if every vertex has
degree at most one. If \(\theta\in V(\mathcal M)\) is paired, denote
its unique partner by \(\theta^*\). Then \((\theta^*)^*=\theta\) and \(\theta^*\neq\theta\). Since \(\theta\) and \(\theta^*\) are distinct one-dimensional
subspaces of \(\mathbb R^2\), we have
\(\mathbb R^2=\theta\oplus\theta^*\).
Hence, for every \(d\in\mathbb R^2\), there is a unique decomposition
\begin{equation}\label{eq:unique-completion}
d=u+v,
\qquad
u\in\theta,
\qquad
v\in\theta^*.
\end{equation}
Consequently, fixing the endpoints of a diagonal and one paired side
direction uniquely determines the decomposition above, and hence the
corresponding parallelogram.


The main result in this subsection is the following.
\begin{lemma}
\label{lem:matched-directions}
Let \(\mathcal M\) be a matching on \(\mathbb{RP}^1\).
Let \(P\subset\mathbb R^2\) be finite and let
\(f:\mathbb R^2\to[0,\infty)\) be supported on \(P\).
Define
\[
 \mathcal R_{\mathcal M}(f)
=\sum_{\substack{(p_1,p_2,p_3,p_4)\in P^4\\
p_1+p_3=p_2+p_4,\ p_i\text{ distinct}\\
\{\dir(p_2-p_1),\dir(p_4-p_1)\}\in E(\mathcal M)}}
\prod_{i=1}^4 f(p_i),
\]
where the vertices are taken in cyclic boundary order. Then there exists an absolute constant \(C>0\) such that
\begin{equation}\label{eq:matched-directions}
\mathcal R_{\mathcal M}(f)
\le
C\log\#P
\left(\sum_{x\in P}f(x)^2\right)^2.
\end{equation}
\end{lemma}

We prove Lemma~\ref{lem:matched-directions} using the auxiliary lemmas below, in which \(\mathcal M\) is fixed throughout.

For the rank-block analysis below, fix pairwise disjoint finite sets
\(S_0,\ldots,S_m\subset\mathbb R^2\), where
\(m\in\mathbb Z_{\ge0}\) and \(\#S_j\le2^j\).
As in Herr--Kwak~\cite[Proposition~3.1]{HerrKwak}, consider
\begin{equation}\label{eq:matched-rank-form}
g=\sum_{j=0}^m\rho_j2^{-j/2}\mathbf1_{S_j},
\qquad
\rho_j\ge0\quad (0\le j\le m),
\qquad
\Lambda=\sum_{j=0}^m\rho_j^2.
\end{equation}
Write \(P=\bigcup_{j=0}^mS_j\),
then \(\sum_{x\in P}g(x)^2\le\Lambda\).

We adapt the classification by line richness in
Herr--Kwak~\cite[Section~3]{HerrKwak}, replacing their orthogonal
direction pairs with pairs of directions matched by \(\mathcal M\).
A \emph{matched cross} is a triple \((x,\ell_1,\ell_2)\), where
\(\ell_1=x+\theta\) and \(\ell_2=x+\theta^*\) for some
\(\{\theta,\theta^*\}\in E(\mathcal M)\).

Fix a sufficiently large absolute integer \(C_0\), and set
\(J_j=\lfloor j/2\rfloor+C_0\).
We classify matched crosses \((x,\ell_1,\ell_2)\) with \(x\in \bigcup_{j=0}^mS_j\)
into three types:
\[
\begin{cases}
\text{Type I},   & a\ge J_j,\\
\text{Type II},  & 1\le a<J_j,\\
\text{Type III}, & a=0,
\end{cases}
\]
where \(j\) is the unique index such that \(x\in S_j\), and
\begin{equation}
a
:=
\left\lfloor
\log_2
\max\bigl\{
\#(\ell_1\cap S_j),\,
\#(\ell_2\cap S_j)
\bigr\}
\right\rfloor
\in\mathbb Z_{\ge0}.
\end{equation}
For an ordered parallelogram \(Q=(p_1,p_2,p_3,p_4)\in P^4\) with four distinct vertices and \(\{\operatorname{dir}(p_2-p_1),\operatorname{dir}(p_4-p_1)\}\in E(\mathcal M)\), we say that the corner at \(p_i\), \(i=1,2,3,4\), is of type \(\alpha\in\{\mathrm{I},\mathrm{II},\mathrm{III}\}\) if the matched cross \(\bigl(p_i,p_i+\operatorname{dir}(p_{i+1}-p_i),p_i+\operatorname{dir}(p_{i-1}-p_i)\bigr)\) is of type \(\alpha\), where the indices are understood modulo four.

For \(\alpha,\beta\in\{\mathrm{II},\mathrm{III}\}\), let \(\mathcal R_{\alpha,\beta}(g)\) denote the contribution to \(\mathcal R_{\mathcal M}(g)\) from ordered parallelograms with corner type pattern \((\alpha,\alpha,\beta,\beta)\) in cyclic order.


To control parallelograms with a type-I corner, we mark the
point--direction pairs corresponding to sufficiently rich side lines.
The following weighted estimate bounds their contribution in terms
of the total marked mass and the full parallelogram energy. This will
allow the type-I contribution to be absorbed into the energy estimate.

For the next lemma only, let \(P\subset\mathbb R^2\) be finite and let
\(f:\mathbb R^2\to[0,\infty)\) be supported on \(P\). Set \(\mu=\sum_{x\in P}f(x)^2.\)
For \(u\in(P-P)\setminus\{0\}\), set
\[
F_u(x)=f(x)f(x+u),
\]
and define
\[
L_u(x,y)
:=
\begin{cases}
\mathbf 1_{\{y-x\parallel(\operatorname{dir}u)^*\}},
& \operatorname{dir}u \text{ is paired in }\mathcal M,\\[2mm]
0,
& \operatorname{dir}u \text{ is unpaired in }\mathcal M .
\end{cases}
\]
By the same argument as before, it is easy to see that \(L_u\) is positive semidefinite. 
We write
\[
\langle F,G\rangle_{L_u}
:=
\sum_{x,y\in P}
F(x)G(y)L_u(x,y).
\]

\begin{lemma}
\label{lem:marked-weighted-estimate}
Let \(\omega:P\times V(\mathcal M)\to[0,1]\)
be a weight, vanishing on unpaired directions and nonzero for only
finitely many directions \(\theta\). Extend \(\omega\) by zero outside
\(V(\mathcal M)\), and define
\begin{align}\label{eq:weight-B-omega}
F_u^\omega(x)
=
\omega(x,\operatorname{dir}u)F_u(x),
\qquad
\mathcal B_\omega(f)
=
\sum_{x\in P}f(x)^2
\sum_{\theta\in V(\mathcal M)}
\omega(x,\theta).    
\end{align}
Then, we have
\begin{equation}\label{eq:marked-weighted-estimate}
\sum_{u\ne0}
\langle F_u^\omega,F_u\rangle_{L_u}
\le
\Bigl[
\bigl(\mathcal R_{\mathcal M}(f)+\mu^2\bigr)
\mu\mathcal B_\omega(f)
\Bigr]^{1/2}.
\end{equation}
\end{lemma}
\begin{proof}
 Splitting the sum into the cases  \(x\neq y\) and \(x=y\), we obtain
\begin{align}\label{eq:matched-completed-form}
\sum_{u\ne0}\langle F_u,F_u\rangle_{L_u}
&=
\mathcal R_{\mathcal M}(f)
+
\sum_{\substack{u\ne0\\
\operatorname{dir}u\ \mathrm{paired}}}
\sum_x f(x)^2f(x+u)^2
\notag\\
&\le
\mathcal R_{\mathcal M}(f)
+
\sum_{x\ne y}f(x)^2f(y)^2
\notag\\
&\le
\mathcal R_{\mathcal M}(f)+\mu^2.
\end{align}

We next show that
\begin{equation}\label{eq:marked-square-bound}
\sum_{u\ne0}
\langle F_u^\omega,F_u^\omega\rangle_{L_u}
\le
\mu\mathcal B_\omega(f).
\end{equation}
Expanding the left hand side and applying the Cauchy--Schwarz
inequality, we obtain
\[
\sum_{u\ne0}
\langle F_u^\omega,F_u^\omega\rangle_{L_u}
\le
(S_1S_2)^{1/2}.
\]
where
\[
S_1
=
\sum_{u,x,y}'
\omega(x,\theta)\omega(y,\theta)
f(x)^2f(y+u)^2,
\]
and
\[
S_2
=
\sum_{u,x,y}'
\omega(x,\theta)\omega(y,\theta)
f(x+u)^2f(y)^2.
\]
where the primed sums are taken over \(u,x,y\) such that
\(u\neq0\), \(\theta=\operatorname{dir}u\) is paired,
\(y-x\parallel\theta^*\), together with the support conditions \(x,x+u,y,y+u\in P\). 

In $S_1$, fixing \(x\), \(z=y+u\), and \(\theta\),
the unique decomposition
property \eqref{eq:unique-completion} leaves at most one pair
\((u,y)\).
Hence, since \(0\le\omega\le1\), we have
\[
S_1
\le
\sum_{x,z,\theta}
f(x)^2f(z)^2\omega(x,\theta)
=
\mu\mathcal B_\omega(f).
\]
The same argument applied to \(S_2\) gives the same bound, proving
\eqref{eq:marked-square-bound}. By the Cauchy--Schwarz inequality and
\eqref{eq:matched-completed-form}, \eqref{eq:marked-square-bound},
we have
\[
\begin{aligned}
\sum_{u\neq0}
\langle F_u^\omega,F_u\rangle_{L_u}
&\le
\left(
\sum_{u\neq0}
\langle F_u^\omega,F_u^\omega\rangle_{L_u}
\right)^{1/2}
\left(
\sum_{u\neq0}
\langle F_u,F_u\rangle_{L_u}
\right)^{1/2}
\\
&\le
\left[
\mu\mathcal B_\omega(f)
\bigl(\mathcal R_{\mathcal M}(f)+\mu^2\bigr)
\right]^{1/2},
\end{aligned}
\]
which is just \eqref{eq:marked-weighted-estimate}.
\end{proof}

We now return to the rank-block function \(g\) and the set
\(P=\bigcup_{j=0}^mS_j\). In the application below, \(F_u\) and
\(F_u^\omega\) are formed with \(f=g\).

For the remaining configurations, we adapt the Cauchy--Schwarz reduction in Herr--Kwak~\cite[Lemma~3.2]{HerrKwak}. Their lemma is formulated for lattice parallelograms and involves an arithmetic divisibility condition. Since \(P\subset\mathbb R^2\) is arbitrary, this divisibility condition is not available in our setting. We therefore use only the geometric part of their argument. More precisely, for each nonzero matched difference \(u\), with
\(\theta=\operatorname{dir}u\), we group the segments
\((x,x-u)\) according to the affine line \(\ell=x+\theta^*\)
and the types of \(x\) and \(x-u\). Applying the Cauchy--Schwarz
inequality, we bound the contribution from parallelograms without
type-I corners by the terms \(\mathcal R_{\alpha,\beta}(g)\),
\(\alpha,\beta\in\{\mathrm{II},\mathrm{III}\}\), together with a
diagonal term bounded by \(\Lambda^2\).

\begin{lemma}
\label{lem:matched-pattern-reduction}
With the notation above, there exists an absolute constant \(C>0\)
such that
\begin{equation}\label{eq:matched-pattern-reduction}
\mathcal R_{\mathcal M}(g)
\le
C\bigl[
\mathcal R_{\mathrm{II},\mathrm{II}}(g)
+
\mathcal R_{\mathrm{III},\mathrm{III}}(g)
+
\mathcal R_{\mathrm{II},\mathrm{III}}(g)
+
\Lambda^2
\bigr].
\end{equation}
\end{lemma}

\begin{proof}
Let \(\mathcal L_j\) be the family of affine lines containing at
least \(2^{J_j}\) points of \(S_j\).
Lemma~\ref{lem:rich-lines} gives
\[
\#\mathcal L_j
\lesssim
2^{2j-3J_j}+2^{j-J_j}
\lesssim
2^{j/2}.
\]
By \eqref{eq:ST}, we obtain
\begin{equation}\label{eq:rich-total-incidences}
\sum_{\ell\in\mathcal L_j}\#(S_j\cap\ell)
\lesssim
(\#S_j)^{2/3}(\#\mathcal L_j)^{2/3}
+\#S_j+\#\mathcal L_j
\lesssim 2^j.
\end{equation}

For \(x\in S_j\) and \(\theta\in V(\mathcal M)\), define
\[
\omega(x,\theta)
:=
\begin{cases}
1,&\theta\text{ is paired and }x+\theta\in\mathcal L_j,\\
0,&\text{otherwise}.
\end{cases}
\]
By the definition of \(\mathcal B_\omega\) (see \eqref{eq:weight-B-omega}) and
\eqref{eq:rich-total-incidences}, we have
\begin{equation}\label{eq:B-omega}
\mathcal B_\omega(g)
\le
\sum_{j=0}^m\rho_j^22^{-j}
\sum_{\ell\in\mathcal L_j}\#(S_j\cap\ell)
\lesssim \Lambda.
\end{equation}

Let \(\mathcal R_{\mathrm{no\,I}}(g)\) denote the contribution to
\(\mathcal R_{\mathcal M}(g)\) from ordered parallelograms with no
type-I corner.  If \(x\) is a type-I corner, then one of its incident side lines belongs to \(\mathcal L_j\).  Choosing \(u\) along this side gives
\(\omega(x,\operatorname{dir}u)=1\), and \(L_u(x,y)=1\) for the other adjacent vertex \(y\).
Hence, using Lemma~\ref{lem:marked-weighted-estimate},
\eqref{eq:B-omega}, \(\sum_x g(x)^2\le\Lambda\), and Young's
inequality, we obtain
\[
\mathcal R_{\mathcal M}(g)-\mathcal R_{\mathrm{no\,I}}(g)
\le 8\sum_{u\ne0}\langle F_u^\omega,F_u\rangle_{L_u}
\le C\Lambda\bigl(\mathcal R_{\mathcal M}(g)+\Lambda^2\bigr)^{1/2}
\le \frac12\mathcal R_{\mathcal M}(g)+C'\Lambda^2.
\]
It follows that
\begin{equation}\label{eq:matched-type-I-absorption}
\mathcal R_{\mathcal M}(g)
\le 2\mathcal R_{\mathrm{no\,I}}(g)+C\Lambda^2.
\end{equation}

Fix \(u\in(P-P)\setminus\{0\}\) such that
\(\theta=\dir u\) is paired in \(\mathcal M\).
Once \(u\) and \(x\) are fixed, the incident directions at both
\(x\) and \(x-u\) are \(\theta\) and \(\theta^*\).
Hence their corner types are determined.
For an affine line \(\ell\parallel\theta^*\), set
\[
C_{\alpha\beta}(u,\ell)
:=
\sum_{\substack{x\in\ell,\ x,x-u\in P\\
x\text{ has type }\alpha,\ x-u\text{ has type }\beta}}
g(x)g(x-u),
\qquad
\alpha,\beta\in\{\mathrm{II},\mathrm{III}\},
\]
where the types are taken with respect to \(\theta,\theta^*\).
In the sums below, \(u\) ranges over nonzero differences with paired
direction, and \(\alpha,\beta\) range over
\(\{\mathrm{II},\mathrm{III}\}\).
Applying the Cauchy--Schwarz inequality and  \(\sum_x g(x)^2\le\Lambda\), we obtain
\[
\begin{aligned}
\mathcal R_{\mathrm{no\,I}}(g)
&\le
\sum_u\sum_{\ell\parallel\theta^*}
\left(\sum_{\alpha,\beta}C_{\alpha\beta}(u,\ell)\right)^2
\le
4\sum_u\sum_{\ell\parallel\theta^*}\sum_{\alpha,\beta}
C_{\alpha\beta}(u,\ell)^2\\
&\le
4\left(
\sum_{\alpha,\beta}\mathcal R_{\alpha,\beta}(g)
+
\sum_u\sum_{x\in P}g(x)^2g(x-u)^2
\right)\\
&\le
4\left(
\sum_{\alpha,\beta}\mathcal R_{\alpha,\beta}(g)
+\Lambda^2
\right).
\end{aligned}
\]
Finally, cyclic relabeling gives
\(\mathcal R_{\mathrm{III},\mathrm{II}}(g)
=\mathcal R_{\mathrm{II},\mathrm{III}}(g)\), so
\[
\mathcal R_{\mathrm{no\,I}}(g)
\le4\bigl[
\mathcal R_{\mathrm{II},\mathrm{II}}(g)
+\mathcal R_{\mathrm{III},\mathrm{III}}(g)
+2\mathcal R_{\mathrm{II},\mathrm{III}}(g)
+\Lambda^2\bigr].
\]
Combining this with \eqref{eq:matched-type-I-absorption} proves
\eqref{eq:matched-pattern-reduction}.
\end{proof}

The next two lemmas adapt the rectangle counting estimates of
Herr--Kwak~\cite[Lemma~3.4 and estimates~(3.17)--(3.18) in
Lemma~3.5]{HerrKwak} to matched directions. In their proofs,
orthogonality determines the adjacent side direction. Here the
matching provides a unique distinct partner instead. Once three
vertices are chosen, the parallelogram identity fixes the fourth.
Together with the rich line bound, these properties give the same
counting exponents for arbitrary finite point sets in \(\mathbb R^2\).

Following the richness decomposition of
Herr--Kwak~\cite[Section~3]{HerrKwak}, we organize the
parallelograms by their block and corner-richness indices.

For \(\mathbf j=(j_1,j_2,j_3,j_4)\in\{0,\ldots,m\}^4\) and
\(\mathbf a=(a_1,a_2,a_3,a_4)\in\mathbb Z_{\ge0}^4\), let
\(q(\mathbf j,\mathbf a)\) denote the number of ordered
parallelograms \((p_1,p_2,p_3,p_4)\) with matched side directions
and four distinct vertices \(p_k\in S_{j_k}\) such that
\begin{equation}\label{eq:ja}
2^{a_k}
\le
\max\bigl\{
\#(S_{j_k}\cap\ell_{k,k+1}),
\#(S_{j_k}\cap\ell_{k,k-1})
\bigr\}
<
2^{a_k+1},
\qquad k=1,2,3,4.
\end{equation}
where \(\ell_{ij}\) is the affine line through \(p_i\) and \(p_j\),
\(a_k\) is the richness index of the corner \(p_k\), and the indices
are understood modulo four.
\begin{lemma}
\label{lem:matched-counting}
Let \(0\le j_i\le m\), \(i=1,2,3,4\), and \(a_3\ge0\) be integers.
The number of ordered parallelograms with matched side directions,
four distinct vertices \(p_i\in S_{j_i}\) in cyclic order, and
\(\#(S_{j_3}\cap\ell_{23})<2^{a_3+1}\) is at most
\(C2^{j_1+j_2+a_3}\), where \(C>0\) is an absolute constant.
In particular,
\begin{equation}\label{eq:matched-pair-count}
q(\mathbf j,\mathbf a)\lesssim2^{j_1+j_2+a_3}.
\end{equation}
\end{lemma}

\begin{proof}
There are at most \(2^{j_1+j_2}\) choices of distinct vertices
\((p_1,p_2)\in S_{j_1}\times S_{j_2}\).
Set \(\theta=\operatorname{dir}(p_2-p_1)\).
If \(\theta\) is unpaired, there is no admissible parallelogram.
Otherwise,
\[
p_3\in\ell_{23}=p_2+\theta^*.
\]
By assumption, there are fewer than \(2^{a_3+1}\) choices of \(p_3\),
which then uniquely determines the parallelogram. Hence there are at most \(O\bigl(2^{j_1+j_2+a_3}\bigr)\) such parallelograms. By the definition of \(q(\mathbf j,\mathbf a)\),
this gives \eqref{eq:matched-pair-count}.
\end{proof}

We next use the argument for selecting rich lines from
Herr--Kwak~\cite[Lemma~3.5]{HerrKwak}
for parallelograms with matched directions.
\begin{lemma}
\label{lem:matched-rich-counting}
Let \(0\le j_i\le m\) and \(a_i\ge0\), \(i=1,2,3,4\), be integers.
Assume that \(1\le a_1<J_{j_1}\). Then
\begin{equation}\label{eq:matched-rich-counts1}
q(\mathbf j,\mathbf a)
\lesssim
2^{2j_1-2a_1+a_2+a_4}.
\end{equation}
\begin{equation}\label{eq:matched-rich-counts2}
q(\mathbf j,\mathbf a)
\lesssim
2^{2j_1-2a_1+a_2+a_3}.
\end{equation}
\end{lemma}

\begin{proof}
By \eqref{eq:rich-lines}, the number of lines \(\ell\) with
\(2^{a_1}\le\#(S_{j_1}\cap\ell)<2^{a_1+1}\) is at most \(O\bigl(2^{2j_1-3a_1}+2^{j_1-a_1}\bigr)
=
O\bigl(2^{2j_1-3a_1}\bigr),\) and for each such \(\ell\), there are \(O(2^{a_1})\) choices of
\(p_1\in \ell\cap S_{j_1}\).
Consequently, the number of such pairs \((p_1,\ell)\) is
\(O(2^{2j_1-2a_1})\).
Once \(p_1\) and \(\ell\) are fixed, the matching determines the
other incident line through \(p_1\).
There are at most two assignments of these two lines to
\(\ell_{12}\) and \(\ell_{14}\).

To obtain a parallelogram counted by \(q(\mathbf j,\mathbf a)\),
the richness conditions at \(p_2\) and \(p_4\) require
\[
\#(S_{j_2}\cap\ell_{12})<2^{a_2+1},
\qquad
\#(S_{j_4}\cap\ell_{14})<2^{a_4+1}.
\]
Therefore, there are \(O(2^{a_2})\) choices of \(p_2\).
Choosing \(p_4\) next gives \(O(2^{a_4})\) choices, and then
the parallelogram is uniquely determined. Hence, we have
\eqref{eq:matched-rich-counts1}.

On the other hand, after \(p_2\) is chosen, the line through \(p_2\)
parallel to \(\ell_{14}\) is fixed. To obtain a parallelogram counted by
\(q(\mathbf j,\mathbf a)\), we require \(\#(S_{j_3}\cap \ell_{23})<2^{a_3+1}.\)
Hence there are \(O(2^{a_3})\) choices of \(p_3\), which then uniquely determines a parallelogram. It follows that
\eqref{eq:matched-rich-counts2} holds.
\end{proof}

After Lemma~\ref{lem:matched-pattern-reduction} absorbing the contribution from type-I corners, we estimate the remaining patterns by the interpolation and summation argument of Herr--Kwak~\cite[proof of Lemma~3.3, Cases~II--IV]{HerrKwak}.

\begin{lemma} \label{lem:matched-rank-blocks}
For \(g\) as in \eqref{eq:matched-rank-form},
\begin{equation}\label{eq:matched-regular}
\mathcal R_{\mathcal M}(g)\le C(m+1)\Lambda^2,
\end{equation}
where \(C>0\) is an absolute constant.
\end{lemma}

\begin{proof}
By Lemma~\ref{lem:matched-pattern-reduction}, it suffices to bound
\(\mathcal R_{\mathrm{II},\mathrm{II}}(g)\),
\(\mathcal R_{\mathrm{III},\mathrm{III}}(g)\), and
\(\mathcal R_{\mathrm{II},\mathrm{III}}(g)\).
Extend \(\rho_j\) by zero outside \(0\le j\le m\).
For every \(\varepsilon>0\), it is easy to see that
\begin{equation}\label{eq:discrete-kernel-sum}
\sum_{j,k\in\mathbb Z}\rho_j\rho_k2^{-\varepsilon|j-k|}
\le C_\varepsilon\Lambda.
\end{equation}
Throughout the proof, let \(j_1,\ldots,j_4\in\{0,\ldots,m\}\), with cyclic indexing modulo \(4\), so that \(j_{k+4}=j_k\).

\emph{Four type-II corners.}
Let \(a_1,\ldots,a_4\) be integers with
\(1\le a_i<J_{j_i}\), \(i=1,2,3,4\).
With the indices understood modulo four, define
\[
A_{k,l}=2j_k-2a_k+a_{k+1}+a_{k+1+l},
\qquad k=1,\ldots,4,\quad l=1,2.
\]
By Lemma~\ref{lem:matched-rich-counting}, we have
\(q(\mathbf j,\mathbf a)\lesssim 2^{A_{k,l}}\) for every
\(k=1,\ldots,4\) and \(l=1,2\).
Taking weighted geometric means of the bounds above with coefficient arrays
\[
\frac1{24}
\begin{pmatrix}
2&3\\
3&4\\
0&6\\
3&3
\end{pmatrix},
\qquad
\frac1{12}
\begin{pmatrix}
1&2\\
1&2\\
3&0\\
1&2
\end{pmatrix},
\]
whose entries are nonnegative and sum to one, we obtain
\[
\begin{aligned}
q(\mathbf j,\mathbf a)
&\lesssim
2^{\frac12\sum_i j_i-\frac1{12}(j_1-j_2)},\\
q(\mathbf j,\mathbf a)
&\lesssim
2^{\frac12\sum_i j_i+\frac16(a_1-a_2)}.
\end{aligned}
\]
Cyclic and reversed cyclic relabelings give both sign choices for each
adjacent difference. Taking the stronger estimate for each sign pair,
and then geometrically averaging the resulting eight bounds, gives
\begin{equation}\label{eq:matched-interpolation}
q(\mathbf j,\mathbf a)
\lesssim
2^{\frac12\sum_i j_i
-\varepsilon\sum_{i=1}^4
\bigl(|j_i-j_{i+1}|+|a_i-a_{i+1}|\bigr)},
\qquad
\varepsilon=\frac1{96}.
\end{equation}

By \eqref{eq:matched-interpolation} and
\eqref{eq:discrete-kernel-sum},
\[
\begin{aligned}
\mathcal R_{\mathrm{II},\mathrm{II}}(g)
&\le
\sum_{\substack{0\le j_i\le m\\
1\le a_i<J_{j_i}\\
i=1,2,3,4}}
q(\mathbf j,\mathbf a)
\prod_{i=1}^4\rho_{j_i}2^{-j_i/2}\\
&\lesssim
\sum_{\substack{0\le j_i\le m\\
1\le a_i<J_{j_i}\\
i=1,2,3,4}}
2^{-\varepsilon\sum_{i=1}^4
\bigl(|j_i-j_{i+1}|+|a_i-a_{i+1}|\bigr)}
\prod_{i=1}^4\rho_{j_i}\\
&\lesssim
\left(
\sum_{j,k\in\mathbb Z}
\rho_j\rho_k2^{-\varepsilon|j-k|}
\right)^2
\sum_{\substack{1\le a_i\le m/2+C_0\\
i=1,2,3,4}}
2^{-\varepsilon
\bigl(|a_1-a_2|+|a_2-a_3|+|a_3-a_4|\bigr)}\\
&\lesssim
(m+1)
\left(
\sum_{j,k\in\mathbb Z}
\rho_j\rho_k2^{-\varepsilon|j-k|}
\right)^2
\lesssim
(m+1)\Lambda^2.
\end{aligned}
\]

\emph{Four type-III corners.}
Here \(a_i=0\) for all \(i\).
Let \(q_{\mathbf j}=q(\mathbf j,(0,0,0,0))\), by \eqref{eq:matched-pair-count}, we have
\begin{equation}\label{eq:mixed-all-three}
q_{\mathbf j}\lesssim2^{\min_i(j_i+j_{i+1})}
=2^{\frac12\sum_i j_i
-\frac12\bigl(|j_1-j_3|+|j_2-j_4|\bigr)}.
\end{equation}
It follows from \eqref{eq:mixed-all-three} and \eqref{eq:discrete-kernel-sum} that
\[
\begin{aligned}
\mathcal R_{\mathrm{III},\mathrm{III}}(g)
&\le
\sum_{0\le j_i\le m}
q_{\mathbf j}\,
2^{-\frac12(j_1+j_2+j_3+j_4)}
\rho_{j_1}\rho_{j_2}\rho_{j_3}\rho_{j_4}\\
&\lesssim
\sum_{0\le j_i\le m}
2^{-\frac12|j_1-j_3|}
\rho_{j_1}\rho_{j_3}\,
2^{-\frac12|j_2-j_4|}
\rho_{j_2}\rho_{j_4}\\
&\lesssim
\Lambda^2.
\end{aligned}
\]

\emph{For the pattern \(\mathrm{II},\mathrm{II},\mathrm{III},\mathrm{III}.\)}
set
\[
q_{\mathbf j}
:=
\sum_{\substack{
1\le a_1<J_{j_1}\\
1\le a_2<J_{j_2}
}}
q\bigl(\mathbf j,(a_1,a_2,0,0)\bigr).
\]
Applying Lemma~\ref{lem:matched-counting}, together with its cyclic and
reversed cyclic versions, directly to the union over all admissible
\(a_1\) and \(a_2\), without fixing them individually, we obtain
\begin{equation}\label{eq:mixed-total-pair-counts}
q_{\mathbf j}
\lesssim2^{\min\{j_1+j_4,j_2+j_3,j_1+j_2\}}.
\end{equation}

For \(\mathbf a=(a_1,a_2,0,0)\), with
\(1\le a_i<J_{j_i}\) for \(i=1,2\),
by \eqref{eq:matched-pair-count} and its reversed cyclic versions, we also have
\begin{equation}\label{eq:mixed-count-a}
q(\mathbf j,\mathbf a)
\lesssim
2^{j_3+j_4+\min\{a_1,a_2\}}
\le
2^{j_3+j_4+(a_1+a_2)/2}.
\end{equation}
The corresponding versions of \eqref{eq:matched-rich-counts2} give
\begin{equation}\label{eq:mixed-count-b}
q(\mathbf j,\mathbf a)
\lesssim\min\{2^{2j_1-2a_1},2^{2j_2-2a_2}\}
\le2^{j_1+j_2-a_1-a_2}.
\end{equation}

Interpolating \eqref{eq:mixed-count-a} and \eqref{eq:mixed-count-b}, we have
\[
q(\mathbf j,\mathbf a)
\lesssim2^{\frac25(j_1+j_2)+\frac35(j_3+j_4)
-\frac1{10}(a_1+a_2)},
\]
which implies
\begin{align} \label{eq:mixed-interpolated-count}
q_{\mathbf j}
&=
\sum_{\substack{
1\le a_1<J_{j_1}\\
1\le a_2<J_{j_2}
}}
q\bigl(\mathbf j,(a_1,a_2,0,0)\bigr)
\notag\\
&\lesssim
\sum_{a_1,a_2\ge1}
2^{\frac25(j_1+j_2)+\frac35(j_3+j_4)-\frac1{10}(a_1+a_2)}
\notag\\
&\lesssim
2^{\frac25(j_1+j_2)+\frac35(j_3+j_4)}.
\end{align}
Combining \eqref{eq:mixed-total-pair-counts} , \eqref{eq:mixed-interpolated-count} and the inequality 
\[
\begin{aligned}
&\min\left\{
j_1+j_4,\,
j_2+j_3,\,
j_1+j_2,\,
\frac25(j_1+j_2)+\frac35(j_3+j_4)
\right\}
-\frac12\sum_{i=1}^4j_i\\
&\hspace{3cm}
\le
-\frac1{10}\bigl(|j_1-j_3|+|j_2-j_4|\bigr),
\end{aligned}
\]
we have
\[
q_{\mathbf j}
\lesssim
2^{\frac12\sum_{i=1}^4j_i
-\frac1{10}\bigl(|j_1-j_3|+|j_2-j_4|\bigr)}.
\]
Therefore,
\[
\begin{aligned}
\mathcal R_{\mathrm{II},\mathrm{III}}(g)
&\le
\sum_{0\le j_i\le m}
q_{\mathbf j}\,
2^{-\frac12\sum_{i=1}^4j_i}
\prod_{i=1}^4\rho_{j_i}\\
&\lesssim
\sum_{0\le j_i\le m}
2^{-\frac1{10}\bigl(|j_1-j_3|+|j_2-j_4|\bigr)}
\prod_{i=1}^4\rho_{j_i}\\
&=
\left(
\sum_{0\le j_1,j_3\le m}
\rho_{j_1}\rho_{j_3}2^{-|j_1-j_3|/10}
\right)
\left(
\sum_{0\le j_2,j_4\le m}
\rho_{j_2}\rho_{j_4}2^{-|j_2-j_4|/10}
\right)\\
&\lesssim
\Lambda^2.
\end{aligned}
\]
Combining the three estimates above with
Lemma~\ref{lem:matched-pattern-reduction}, we complete the proof of
\eqref{eq:matched-regular}.
\end{proof}

We now complete the proof of Lemma~\ref{lem:matched-directions}.
Following Herr--Kwak~\cite[proof of Theorem~1.2]{HerrKwak},
we order the support by decreasing values of \(f\) and divide it
into dyadic blocks.
Their argument requires each point to lie on at most one sufficiently
rich line, which is achieved by iteratively removing exceptional points.
Lemma~\ref{lem:matched-rank-blocks} requires no such assumption.
Hence this iterative removal is unnecessary here.

\begin{proof}[Proof of Lemma~\ref{lem:matched-directions}]

The case \(f=0\) is immediate. Otherwise, let
\(N=\#(\operatorname{supp}f)\le\#P\) and enumerate the
support so that
\[
f(x_1)\ge f(x_2)\ge\cdots\ge f(x_N)>0.
\]
For \(0\le j\le m=\lfloor\log_2N\rfloor\), set
\[
S_j=\{x_k:2^j\le k<2^{j+1},\ k\le N\},
\qquad
\rho_j=2^{j/2}f(x_{2^j}),
\]
and define
\[
g=\sum_{j=0}^m\rho_j2^{-j/2}\mathbf1_{S_j}.
\]
The blocks are disjoint, \(\#S_j\le2^j\), and \(f\le g\).
For \(j\ge1\), by the monotonicity of \(f(x_k)\), we have
\[
\rho_j^2=2^j f(x_{2^j})^2
\le2\sum_{k=2^{j-1}}^{2^j-1}f(x_k)^2.
\]
Including \(\rho_0^2=f(x_1)^2\), we obtain
\[
\sum_{j=0}^m\rho_j^2\le3\sum_{x\in P}f(x)^2,
\qquad
m+1\lesssim\log\#P.
\]
It follows from Lemma~\ref{lem:matched-rank-blocks} that
\[
\mathcal R_{\mathcal M}(f)
\le
\mathcal R_{\mathcal M}(g)
\lesssim
(m+1)\left(\sum_{j=0}^m\rho_j^2\right)^2
\lesssim
\log\#P
\left(\sum_{x\in P}f(x)^2\right)^2.
\]
This proves \eqref{eq:matched-directions}.
\end{proof}

\subsection{Bounded-degree graphs and fixed length scales}

A \emph{proper edge coloring} colors the edges of a finite graph so that
any two edges sharing a vertex have different colors.
Hence, the edges of any fixed color form a matching.
If the graph has maximum degree at most \(d\ge1\), then its edges
can be colored with at most \(2\lfloor d\rfloor-1\le2d\)
colors. Indeed, for each edge, at most \(2(\lfloor d\rfloor-1)\) colors are already used by edges sharing one of its endpoints, so at least one color
remains available.

\begin{lemma}\label{lem:weighted-graph}
Let \(\alpha\in I\), and let \(\mathcal G\) have maximum degree at most
\(d\ge1\). If \(f_1,f_2,f_3,f_4\ge0\) are supported in a finite
nonempty set \(S\subset\mathbb Z^2\), then
\begin{equation}\label{eq:weighted-graph}
 \mathcal T_\cG(f_1,f_2,f_3,f_4)
 \le Cd\log\#S\prod_{i=1}^4\norm{f_i}_{\ell^2}.
\end{equation}
\end{lemma}

\begin{proof}
Let \(f\) be supported in \(S\). Only directions in the following finite set
can appear in a nonzero term of \(\mathcal E_{\mathcal G}(f)\):
\[
\{\operatorname{dir}(T(x-y)):x,y\in S,\ x\ne y\}
.\]
Restrict the graph to these
vertices. By the edge coloring argument above, its edge set can be partitioned into at most \(\lfloor 2d\rfloor\) matchings; if there are no edges the assertion
below is immediate.

For each matching \(\mathcal M\), set \(P=TS\) and define
\(\widetilde f(Tx)=f(x)\) for \(x\in S\).
Since \(T\) is invertible,
\[
\#P=\#S,
\qquad
\sum_{p\in P}\widetilde f(p)^2
=
\sum_{x\in S}f(x)^2.
\]
For each \(\mathcal M\), the \(T\)-image of a parallelogram \((x,x+u,x+u+v,x+v)\)
has adjacent side directions \(\operatorname{dir}(Tu)\) and
\(\operatorname{dir}(Tv)\). Since \(0\le K\le1\),
Lemma~\ref{lem:matched-directions} gives
\[
\mathcal E_{\mathcal M}(f)
\le
\mathcal R_{\mathcal M}(\widetilde f)
\lesssim
\log\#S\|f\|_{\ell^2}^4.
\]
Summing over at most \(2d\) matchings, we obtain
\[
\mathcal E_{\mathcal G}(f)
\lesssim
d\log\#S\|f\|_{\ell^2}^4.
\]
Hence
\[
H_{\mathcal G}(f)
=
\mathcal E_{\mathcal G}(f)+d\|f\|_{\ell^2}^4
\lesssim
d\log\#S\|f\|_{\ell^2}^4.
\]
Applying Lemma~\ref{lem:graph-holder} to \(f_1,\ldots,f_4\)
completes the proof.
\end{proof}

We next use the elementary identity
\begin{equation}\label{eq:determinant-quantization}
\det(Tu,Tv)=\sqrt{\alpha}\det(u,v),\qquad
\det(u,v)=u_1v_2-u_2v_1\in\mathbb Z.
\end{equation}
Thus every nonzero transformed determinant has absolute value at least
\(\sqrt{\alpha_-}\).

\begin{lemma}\label{lem:fixed-scale-directions}
Fix $\alpha\in I$ and $R\ge \rho\ge \lambda$.
Call an ordered pair $(u,v)$ of nonzero vectors in $\mathbb Z^2$
admissible if
$\operatorname{dir}(Tu)\ne \operatorname{dir}(Tv)$ and
\begin{equation}
R\le |Tu|<2R,
\qquad
\rho\le |Tv|<2\rho,
\qquad
|B(u,v)|\le \eta.
\end{equation}
Let $G_{R,\rho}$ be the simple undirected graph on the
unoriented directions in the transformed plane with edge set
\[
E(G_{R,\rho})
=
\Bigl\{
\{\operatorname{dir}(Tu),\operatorname{dir}(Tv)\}
:
(u,v)\text{ is admissible}
\Bigr\}.
\]
Then $G_{R,\rho}$ has maximum degree at most
\[
C_I\left(1+\frac{R}{\rho}\right).
\]
\end{lemma}

\begin{proof}
To distinguish the two length scales, take two copies of the
unoriented direction space, one for the \(R\)-scale and one for the
\(\rho\)-scale. Each edge joins a direction in the \(R\)-scale copy
to a direction in the \(\rho\)-scale copy.

Fix a direction \(\theta\) in the \(R\)-scale copy, and let
\(e\) be a unit vector in \(\theta\). For every admissible pair
\((u,v)\) with \(\operatorname{dir}(Tu)=\theta\), we have
\[
 |e\cdot Tv|\le\eta/R,\qquad |Tv|<2\rho.
\]
Let \((u,v)\) and \((u',v')\) be two admissible pairs with \(\operatorname{dir}(Tu)=\operatorname{dir}(Tu')=\theta\). Let \(e^\perp\) be a unit vector perpendicular to \(e\). It is easy to see that
\[
\begin{aligned}
 |\det(Tv,Tv')|
 &=\bigl|(e\cdot Tv)(e^\perp\cdot Tv')
        -(e^\perp\cdot Tv)(e\cdot Tv')\bigr|\\
 &\le\frac{4\eta\rho}{R}
 \le4\eta<\sqrt{\alpha_-},
\end{aligned}
\]
where we used $\rho\le R$ and the definition of $\eta$.
By \eqref{eq:determinant-quantization}, this determinant must be zero.
Hence \(\operatorname{dir}(Tv)=\operatorname{dir}(Tv')\).
Thus there is at most one direction in the \(\rho\)-scale copy adjacent to \(\theta\).

Now fix a direction \(\phi\) in the \(\rho\)-scale copy, and let
\(e\) be a unit vector in \(\phi\). For every admissible pair
\((u,v)\) with \(\operatorname{dir}(Tv)=\phi\), we have
\[
\left|e\cdot\frac{Tu}{|Tu|}\right|
\le\frac{\eta}{R\rho}.
\]
Since the directions are unoriented, we understand angles modulo
\(\pi\). Consequently, \(\operatorname{dir}(Tu)\) lies in an angular
interval of length at most \(C\eta/(R\rho)\) centered at
\(\operatorname{dir}(e^\perp)\), the interval is very short because
$\eta/(R\rho)\le\eta/\lambda^2\le10^{-6}$.
Let \((u,v)\) and \((u',v')\) be two admissible pairs with
\(\operatorname{dir}(Tv)=\operatorname{dir}(Tv')=\phi\) and
\(\operatorname{dir}(Tu)\ne\operatorname{dir}(Tu')\). 
Since \(|Tu|,|Tu'|<2R\), by \eqref{eq:determinant-quantization} we have
\[
\left|\sin\angle(Tu,Tu')\right|
=\frac{|\det(Tu,Tu')|}{|Tu||Tu'|}
\ge\frac{\sqrt{\alpha_-}}{4R^2}.
\]
Since \(|\sin t|\le |t|\), the angular separation of \(\operatorname{dir}(Tu)\) and \(\operatorname{dir}(Tu')\) is at least \(\sqrt{\alpha_-}/(4R^2)\).
An interval of length $A$ contains at most $1+A/\delta$ points
separated by at least $\delta$. Therefore, the number of \(R\)-scale directions is at most
\[
1+\frac{C\eta/(R\rho)}{\sqrt{\alpha_-}/(4R^2)}
\le C_I\left(1+\frac{R}{\rho}\right).
\]

The geometric intuition behind the two cases is illustrated in
Figure~\ref{fig:fixed-scale-directions}.
\begin{figure}[H]
\centering
\resizebox{0.85\linewidth}{!}{%
\begin{tikzpicture}[
    x=1cm,y=1cm,>=latex,
    line width=0.45pt,line cap=round,line join=round,
    font=\footnotesize,
    every node/.style={inner sep=2pt},
    title/.style={font=\footnotesize},
    guide/.style={black!40,line width=0.3pt},
    vector/.style={->,black,line width=0.6pt},
    hypothetical/.style={black!65,dashed,line width=0.5pt},
    dim/.style={<->,black!65,line width=0.4pt}
]

\node[title] at (0.90,5.05)
    {(a) Fixed $R$-scale direction $\theta$};
\coordinate (Oa) at (0,0);
\coordinate (V) at (0.12,3.15);
\coordinate (Vp) at (-0.27,2.25);

\fill[black!4] (-0.34,0) rectangle (0.34,3.5);
\draw[guide,dashed] (-0.34,0)--(-0.34,3.5);
\draw[guide,dashed] (0.34,0)--(0.34,3.5);
\fill[black!12] (Oa)--(V)--(Vp)--cycle;
\draw[guide,densely dotted,->] (0,0)--(0,3.75);
\node[anchor=west] at (0.52,3.72) {$e^\perp$};
\draw[->] (-1.45,0)--(3.25,0)
    node[below right] {$e\parallel\theta$};
\draw[guide] (0.16,0)--(0.16,0.16)--(0,0.16);

\draw[vector] (Oa)--(V);
\draw[hypothetical,->] (Oa)--(Vp);
\draw[hypothetical] (Vp)--(V);
\fill (V) circle[radius=1.5pt];
\draw[hypothetical,fill=white] (Vp) circle[radius=1.7pt];
\node[anchor=west] at (0.40,3.12) {$Tv$};
\node[anchor=east] at (-0.48,2.25) {$Tv'$};
\fill (Oa) circle[radius=1.4pt];
\node[below left] at (Oa) {$O$};

\draw[guide] (-0.34,3.5)--(-0.34,4.05);
\draw[guide] (0.34,3.5)--(0.34,4.05);
\draw[dim] (-0.34,3.98)--(0.34,3.98)
    node[midway,above=3pt] {$2\eta/R$};
\draw[guide] (0.34,0)--(1.85,0);
\draw[guide] (0,3.5)--(1.85,3.5);
\draw[dim] (1.72,0)--(1.72,3.5)
    node[midway,right=3pt] {$2\rho$};

\node at (0.90,-0.78)
    {$\displaystyle
      |\triangle O\,Tv\,Tv'|
      \le 2\eta\rho/R< \sqrt{\alpha_-}/2$};

\begin{scope}[shift={(7.15,0)}]
    \node[title] at (0.90,5.05)
        {(b) Fixed $\rho$-scale direction $\phi$};
    \coordinate (Ob) at (0,0);
    \coordinate (P1) at (75:2.60);
    \coordinate (P2) at (85:3.20);
    \coordinate (P3) at (100:2.60);
    \coordinate (P4) at (105:3.10);

    \fill[black!4] (Ob)--(72:3.5)
        arc[start angle=72,end angle=108,radius=3.5]--cycle;
    \fill[black!14] (Ob)--(P1)--(P2)--cycle;
    \fill[black!8] (Ob)--(P2)--(P3)--cycle;
    \fill[black!14] (Ob)--(P3)--(P4)--cycle;
    \draw[guide] (Ob)--(72:3.5)
        arc[start angle=72,end angle=108,radius=3.5]--cycle;
    \draw[guide,densely dotted,->] (0,0)--(0,3.75);
    \node[anchor=west] at (0.52,3.72) {$e^\perp$};
    \draw[->] (-1.45,0)--(3.25,0)
        node[below right] {$e\parallel\phi$};
    \draw[guide] (0.16,0)--(0.16,0.16)--(0,0.16);
    \draw (P1)--(P2)--(P3)--(P4);
    \foreach \p in {P1,P2,P3,P4}{
        \draw[vector] (Ob)--(\p);
        \fill (\p) circle[radius=1.5pt];
    }
    \node[anchor=west,xshift=6pt,yshift=2pt] at (P1) {$P_1$};
    \node[anchor=west,xshift=4pt] at (P2) {$P_2$};
    \node[anchor=north west,yshift=-2pt] at (P3) {$P_3$};
    \node[anchor=east,xshift=-6pt,yshift=2pt] at (P4) {$P_4$};
    \fill (Ob) circle[radius=1.4pt];
    \node[below left] at (Ob) {$O$};

    \draw[guide] (72:3.58)--(72:4.10);
    \draw[guide] (108:3.58)--(108:4.10);
    \draw[dim] (72:3.98)
        arc[start angle=72,end angle=108,radius=3.98]
        node[midway,above=3pt] {$\beta\le C\eta/(R\rho)$};
    \draw[guide] (0,0)--(1.85,0);
    \draw[guide] (0,3.5)--(1.85,3.5);
    \draw[dim] (1.72,0)--(1.72,3.5)
        node[midway,right=3pt] {$2R$};

    \node at (0.90,-0.78)
        {$\displaystyle
          P_i=Tu_i,\qquad
          |\triangle OP_iP_{i+1}|\ge \sqrt{\alpha_-}/2$};
\end{scope}
\end{tikzpicture}%
}
\caption{Geometry for Lemma~\ref{lem:fixed-scale-directions}
(schematic).
(a) A fixed $R$-scale direction leaves at most one \(\rho\)-scale
partner direction: the hypothetical triangle would have area below the lattice
minimum. (b) For a fixed $\rho$-scale direction, the \(R\)-scale partner directions lie in a short angular interval with a minimum separation.
Representatives are chosen in the upper half-plane;
$|\triangle|$ denotes area.}
\label{fig:fixed-scale-directions}
\end{figure}
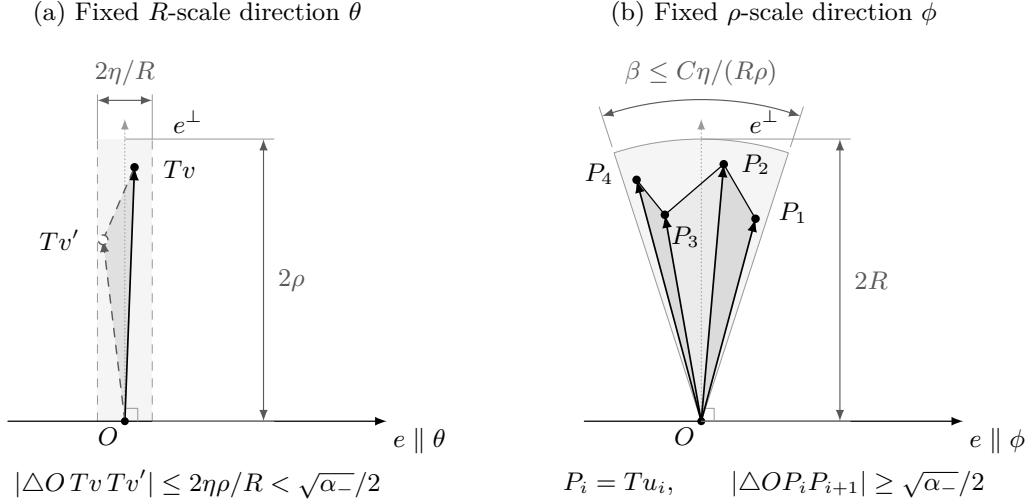

Finally, identify the \(R\)-scale and \(\rho\)-scale copies with the
original unoriented direction space. For any direction \(\theta\),
every direction adjacent to \(\theta\) in \(\mathcal G_{R,\rho}\)
arises from one of the two copies. Hence, by the two estimates above,
\[
\deg_{\mathcal G_{R,\rho}}(\theta)
\le C_I\left(1+\frac{R}{\rho}\right).
\]
This completes the proof of Lemma~\ref{lem:fixed-scale-directions}.
\end{proof}

\subsection{Summation over long point pairs}
\label{sec:aspect}

For $s\ge1$, let $\mathcal R_s(f)$ be the part of
$\cE_\alpha^*(f)$ with
\[
 s\le\frac{|Tv|}{|Tu|}<2s.
\]
We call the ratio of the longer transformed side length \(|Tv|\)
to the shorter transformed side length \(|Tu|\) the \emph{aspect ratio}. The localization argument below allows the dyadic side-length scales
to be summed without introducing a factor depending on the diameter
of the support.

\begin{proposition}\label{prop:aspect}
Let $f\ge0$ be supported in a finite nonempty set $S\subset\Z^2$.
For every $s\ge1$,
\begin{equation}\label{eq:aspect-bound}
 \mathcal R_s(f)
 \le C_Is\log\#SJ(f).
\end{equation}
\end{proposition}

\begin{proof}
Split the summation defining \(\mathcal R_s(f)\) by the dyadic scale of the longer transformed side \(|Tv|\), namely \(R\le |Tv|<2R\), where \(R=\lambda2^j\), \(j\ge0\). Within this range, the shorter transformed side satisfies
\[
\frac{R}{2s}<|Tu|<\frac{2R}{s}.
\]
This interval intersects only a bounded number of dyadic intervals $\rho\le|Tu|<2\rho$, where $\rho=\lambda2^k$, $k\ge0$.
For the boundedly many relevant values of \(\rho\), let
\(\mathcal G_{R,\rho}\) be the graphs from
Lemma~\ref{lem:fixed-scale-directions}. Their union, denoted by
\(\mathcal G_{R,s}\), has maximum degree at most \(C_I s\).

Partition the transformed plane into half open squares of side
$r=R/16$, namely
$[k_1r,(k_1+1)r)\times[k_2r,(k_2+1)r)$ for
$(k_1,k_2)\in\Z^2$. For each square \(Q\), define
\[
f_Q(x)=f(x)\mathbf 1_Q(Tx),
\qquad
\mu_Q=\|f_Q\|_{\ell^2}^2,
\]
and let \(c(Q)\) denote its center. A quadruple of squares \((Q_1,Q_2,Q_3,Q_4)\) is called \emph{compatible}
if there exists at least one parallelogram contributing to the current
\((R,s)\)-scale sum, with \(Tx,T(x+u),T(x+v),T(x+u+v)\) lying in
\(Q_1,Q_2,Q_3,Q_4\), respectively.

For a compatible quadruple \((Q_1,Q_2,Q_3,Q_4)\), since \(R\le |Tv|<2R\) and the distance from any point in a square
to its center is at most \(r/\sqrt2\), both pairs
\((Q_1,Q_3)\) and \((Q_2,Q_4)\) satisfy
\begin{equation}\label{eq:long-square-pair}
 R-\sqrt2r\le|c(Q)-c(Q')|\le2R+\sqrt2r.
\end{equation}
The bound \(|Tu|<2R\) then implies that both pairs
\((Q_1,Q_2)\) and \((Q_3,Q_4)\) satisfy
\begin{equation}\label{eq:short-square-pair}
|c(Q)-c(Q')|<2R+\sqrt2\,r.
\end{equation}
Hence, once \(Q_1\) and \(Q_3\) are fixed, the centers \(c(Q_2)\) and \(c(Q_4)\) lie in disks of radius \(2R+\sqrt2\,r\) centered at \(c(Q_1)\) and \(c(Q_3)\), respectively. Since $r=R/16$, \eqref{eq:short-square-pair} leaves $O(1)$ choices of
$(Q_2,Q_4)$ for each fixed $(Q_1,Q_3)$, and conversely.

Let \(\mathcal C_{R,s}(Q_1,Q_2,Q_3,Q_4)\) denote the sum of the
contributions of all parallelograms in the current \((R,s)\)-scale
class whose transformed vertices lie in \(Q_1,Q_2,Q_3,Q_4\),
respectively. Hence, by dropping the length restrictions and applying
Lemma~\ref{lem:weighted-graph}, we have
\begin{align}\label{eq:scale-pair-bound}
&\sum_{\substack{
x,u,v\in\mathbb Z^2,\ u,v\ne0\\
R\le |Tv|<2R\\
s\le |Tv|/|Tu|<2s}}
f(x)f(x+u)f(x+v)f(x+u+v)K(B_\alpha(u,v))
\notag\\
&\qquad=
\sum_{\substack{
(Q_1,Q_2,Q_3,Q_4)\\
\mathrm{compatible}}}
\mathcal C_{R,s}(Q_1,Q_2,Q_3,Q_4)
\notag\\
&\qquad\le
\sum_{\substack{
(Q_1,Q_2,Q_3,Q_4)\\
\mathrm{compatible}}}
\mathcal T_{\mathcal G_{R,s}}
(f_{Q_1},f_{Q_2},f_{Q_3},f_{Q_4})
\notag\\
&\qquad\le
C_I s\log\#S
\sum_{\substack{
(Q_1,Q_2,Q_3,Q_4)\\
\mathrm{compatible}}}
\sqrt{\mu_{Q_1}\mu_{Q_2}\mu_{Q_3}\mu_{Q_4}}
\notag\\
&\qquad\le
\frac{C_I s}{2}\log\#S
\sum_{\substack{
(Q_1,Q_2,Q_3,Q_4)\\
\mathrm{compatible}}}
\bigl(
\mu_{Q_1}\mu_{Q_3}
+
\mu_{Q_2}\mu_{Q_4}
\bigr)
\notag\\
&\qquad\le
C_I s\log\#S
\sum_{\substack{
(Q,Q')\\
\text{satisfying }\eqref{eq:long-square-pair}}}
\mu_Q\mu_{Q'}.
\end{align}

For \(Tx\in Q\) and \(Ty\in Q'\), with \((Q,Q')\) satisfying
\eqref{eq:long-square-pair}, it is easy to see that
\[
\left(1-\frac{\sqrt2}{8}\right)R
\le |T(x-y)|
\le
\left(2+\frac{\sqrt2}{8}\right)R.
\]
In particular, \(x\ne y\).
For a fixed distinct ordered pair \((x,y)\), any scale \(R\) contributing
to the last sum in \eqref{eq:scale-pair-bound} must satisfy
\[
\frac{|T(x-y)|}{2+\sqrt2/8}
\le R\le
\frac{|T(x-y)|}{1-\sqrt2/8}.
\] Since \(R\) takes only the dyadic values \(\lambda2^j\), \(j\ge0\),
the number of such values lying in this interval is bounded by an absolute constant. Thus
\[
\sum_{R=\lambda2^j}
\sum_{\substack{(Q,Q')\\
\text{satisfying }\eqref{eq:long-square-pair}}}
\mu_Q\mu_{Q'}
\le
C\sum_{x\ne y}f(x)^2f(y)^2
=
CJ(f).
\]
Summing \eqref{eq:scale-pair-bound} proves
\eqref{eq:aspect-bound}.
\end{proof}

For $u,v\ne0$, write the symmetric aspect ratio as
\[
 \sigma(u,v)=\frac{\max\{|Tu|,|Tv|\}}{\min\{|Tu|,|Tv|\}}.
\]
The next two bounds will control moderate aspect ratios and the
coverage error for large aspect ratios, respectively.

\begin{corollary}\label{cor:aspect-sums}
Let $f\ge0$ be supported in a finite nonempty set $S$, and let
$\tau\ge2$. The contribution to $\cE_\alpha^*(f)$ with
$\sigma(u,v)\le\tau$ is at most
\begin{equation}\label{eq:moderate-bound}
 C_I\tau\log\#SJ(f).
\end{equation}
Also,
\begin{equation}\label{eq:aspect-error}
 \sum_{\substack{x,u,v\in\Z^2\\u,v\ne0\\\sigma(u,v)>\tau}}
 \frac{f(x)f(x+u)f(x+v)f(x+u+v)K(B(u,v))}{\sigma(u,v)^2}
 \le\frac{C_I}{\tau}\log\#SJ(f).
\end{equation}
\end{corollary}

\begin{proof}
By symmetry in \(u\) and \(v\), it suffices, up to an absolute factor, to consider \(|Tv|\ge |Tu|.\)  For dyadic $s=2^k\ge1$, Proposition~\ref{prop:aspect} bounds the
class $s\le|Tv|/|Tu|<2s$ by
$C_Is\log\#SJ(f)$.
The classes meeting $\sigma\le\tau$ have $s\le\tau$, and
$\sum_{s=2^k\le\tau}s\le2\tau$, which proves
\eqref{eq:moderate-bound}.
The classes meeting $\sigma>\tau$ have $s>\tau/2$.
On such a class $\sigma^{-2}\le s^{-2}$, so the sum of their
bounds is at most a constant times
\[
 \log\#SJ(f)\sum_{\substack{s=2^k\\s>\tau/2}}s^{-1}
 \le\frac{C}{\tau}\log\#SJ(f).
\]
This proves \eqref{eq:aspect-error}.
\end{proof}

\section{Localization and entropy estimates}
\label{sec:localization}

We prove the localization estimates stated in
Subsection~\ref{sec:recursion}, using the notation and parameters
fixed there. We first establish the exact mass identity and use it
to prove entropy dissipation. We then combine the coverage and
localized Gram estimates to prove the energy recursion.

\subsection{Exact mass and entropy dissipation}
\label{subsec:mass-entropy}

\begin{proof}[Proof of Lemma~\ref{lem:exact-mass}]
For notational convenience, we use the language of random
variables and expectations. Let $Z_1,Z_2$ be independent random vectors with probability
density $\phi^2$, and put $\Xi=Z_2-Z_1$. Define
\begin{equation}\label{eq:kappa}
 \kappa_L=\frac12\E\log
       \frac{4-|\Xi|^2/L^2}{1-|\Xi|^2/L^2},
\end{equation}
where $\E$ denotes expectation with respect to the joint law
of $(Z_1,Z_2)$.

The support of $\phi$ gives $|\Xi|\le2$. Since $L>2$, we have
$\log2\le \kappa_L<\infty.$

For $x=y$, a nonzero integrand would require $|z-w|\le 2r < R$, which is incompatible with $(z,w)\in \mathcal A_R$. Thus \eqref{eq:pairwise-exact-mass} holds.

For $x\neq y$, write
$d=|T(y-x)|>0$ and $n=T(y-x)/d$.
The change of variables
\[
 z=Tx+rZ_1,\qquad w=Ty+rZ_2,\qquad t=d/R
\]
turns the integral in \eqref{eq:pairwise-exact-mass} into 
\begin{equation}\label{eq:inner-inte-1}
\int_0^\infty\iint_{\cA_R}
       \Phi_z(x)^2\Phi_w(y)^2\,
       d\gamma_r(z)d\gamma_r(w)\frac{dR}{R}=    \E\int_0^\infty
       \1_{\{1\le|tn+\Xi/L|<2\}}\frac{dt}{t}. 
\end{equation}

For $c=1,2$, the unique positive
root of $|tn+\Xi/L|=c$ is
\[
 t_c(\Xi)=-\frac{\Xi\cdot n}{L}
 +\sqrt{c^2-\frac{|\Xi|^2-(\Xi\cdot n)^2}{L^2}}.
\]
It follows that
\begin{equation}\label{eq:inner-inte-2}
\int_0^\infty
       \1_{\{1\le|tn+\Xi/L|<2\}}\frac{dt}{t}=\log\frac{t_2(\Xi)}{t_1(\Xi)} .   
\end{equation}

Swapping $Z_1,Z_2$ shows that $\Xi$ and $-\Xi$ have the same
distribution. The explicit roots satisfy
\[
 t_c(\Xi)t_c(-\Xi)=c^2-\frac{|\Xi|^2}{L^2}.
\]
Average the logarithmic integral at $\Xi$ and $-\Xi$ to obtain
\[
 \E\log\frac{t_2(\Xi)}{t_1(\Xi)}
 =\frac12\E\log
       \frac{t_2(\Xi)t_2(-\Xi)}{t_1(\Xi)t_1(-\Xi)}
 =\kappa_L.
\]
This, together with \eqref{eq:inner-inte-1} and \eqref{eq:inner-inte-2}, gives \eqref{eq:pairwise-exact-mass}.

Finally, expand $\mu_z\mu_w$ in \eqref{eq:center-mass}, multiply the
pairwise identity \eqref{eq:pairwise-exact-mass} by $f(x)^2f(y)^2$, and sum over $x,y\in \Z^2$ to get
\eqref{eq:exact-mass}.
\end{proof}

The exact mass identity also determines the marginals of the
localization measure, which we now use to prove entropy dissipation.

\begin{proof}[Proof of Lemma~\ref{lem:entropy-dissipation}]
Assume $f\ne0$ and put $\mu=\mu(f)$, $S=\supp f$,
$p(x)=f(x)^2/\mu$. For positive local masses, let
$q_z(x)=f_z(x)^2/\mu_z$, see \eqref{eq:muz-fz} for the definitions of $f_z,\mu_z$. Introduce the measure
\begin{equation}\label{eq:entropy-mixture-measure}
	d\varrho(R,z,w)
	=\frac{\mu_z\mu_w}{\kappa_L\mu^2}
	\1_{\cA_R}\,d\gamma_r(z)d\gamma_r(w)\frac{dR}{R}.
\end{equation}
It records the normalized mass of each separated window pair
and is symmetric in $z,w$. By \eqref{eq:exact-mass}, we have
\begin{equation}\label{eq:mixture-mass}
	\int d\varrho=\frac{J(f)}{\mu^2}
	=1-\sum_{x\in S}p(x)^2.
\end{equation}
For any fixed $x \in S$, expanding $\mu_w$ and applying
\eqref{eq:pairwise-exact-mass} yields
\begin{align}\label{eq:mixture-marginal}
	\int q_z(x)\,d\varrho
	&=\frac{f(x)^2}{\kappa_L\mu^2}\sum_{y\in S}f(y)^2
	\int_0^\infty\iint_{\cA_R}\Phi_z(x)^2\Phi_w(y)^2
	\,d\gamma_r(z)d\gamma_r(w)\frac{dR}{R}\notag\\
	&=p(x)(1-p(x)).
\end{align}

For $\varrho$-almost every $(R,z,w)$, both local masses are
positive. Put $A=\supp q_z$ and $B=\supp q_w$.
Since $|w-z|\ge R>2r$, the supports $A$ and $B$ are disjoint.

Their equal mixture $q_{z,w}=(q_z+q_w)/2$ therefore equals
$q_z/2$ on $A$ and $q_w/2$ on $B$, and vanishes elsewhere.
Since $\sum_{x\in A}q_z(x)=\sum_{x\in B}q_w(x)=1$,
the definition of entropy gives
\begin{align}\label{eq:entrop-identity}
		\operatorname{Ent}(q_{z,w})
		&=\frac12\sum_{x\in A}q_z(x)\log\frac{2}{q_z(x)}
		+\frac12\sum_{x\in B}q_w(x)\log\frac{2}{q_w(x)}\nonumber\\
		&=\log2+\frac12\operatorname{Ent}(q_z)
		+\frac12\operatorname{Ent}(q_w).
\end{align}
By symmetry and \eqref{eq:mixture-marginal},
\begin{equation}\label{eq:entropy-mixture}
	p=\int q_{z,w}\,d\varrho+\sum_{x\in S}p(x)^2\delta_x,
\end{equation}
where $\delta_x$ is the point mass at $x$.
By \eqref{eq:mixture-mass}, the total weight in
\eqref{eq:entropy-mixture} is
\[
\int d\varrho+\sum_{x\in S}p(x)^2=1.
\]
Since all weights are nonnegative and each $q_{z,w}$ and
$\delta_x$ is a probability distribution,
\eqref{eq:entropy-mixture} expresses $p$ as a convex mixture
of probability distributions. Entropy is concave, and
$\operatorname{Ent}(\delta_x)=0$, so the preceding entropy
identity \eqref{eq:entrop-identity} and symmetry in $z,w$ give
\begin{align*}
	\operatorname{Ent}(p)
	&\ge\int\operatorname{Ent}(q_{z,w})\,d\varrho
+\sum_{x\in S}p(x)^2\operatorname{Ent}(\delta_x)=\int\operatorname{Ent}(q_{z,w})\,d\varrho\\
	&=(\log2)\frac{J(f)}{\mu^2}
	+\int\operatorname{Ent}(q_z)\,d\varrho.
\end{align*}
Multiplying by $\mu^2$ and using
$\mu^2\int\operatorname{Ent}(q_z)\,d\varrho
=(\mathscr L\mathscr H)(f)$ proves
\eqref{eq:entropy-dissipation}.

Apply $\mathscr L^j$ to this inequality and sum over
$j=0,\ldots,n-1$. Positivity gives
\[
(\log2)\sum_{j=0}^{n-1}\mathscr L^jJ
\le\mathscr H-\mathscr L^n\mathscr H\le\mathscr H,
\]
which proves \eqref{eq:entropy-telescoping}.
\end{proof}

\subsection{The energy recursion}
\label{subsec:energy-recursion-proof}

To prove Lemma~\ref{lem:energy-recursion}, we first recover thin
interactions by averaging over localization windows, and then bound
the resulting mixed terms by local self-energies.

For a parallelogram with short side $u$, the product
$\Phi_z(x)\Phi_z(x+u)$ records the weight retained when its first
short side is localized at $z$. Localizing the opposite short
side at $w$ and integrating over separated centers gives the
coverage weight
\begin{equation}\label{eq:coverage-kernel}
 \mathcal H_R(x,u,v)
 =\iint_{\cA_R}
 \Phi_z(x)\Phi_z(x+u)\Phi_w(x+v)\Phi_w(x+u+v)
 \,d\gamma_r(z)d\gamma_r(w).
\end{equation}
Only window pairs containing the two respective short sides contribute. For a fixed parallelogram, the integration over these center pairs multiplies its original energy contribution by $ \mathcal H_R(x,u,v)$. Cauchy-Schwarz in each center and $\int \phi^2=1$ give $0\leq \mathcal H_R\leq 1$. The next estimate shows that integration on the long side scale recovers the same mass $\kappa_L$ as in \eqref{eq:pairwise-exact-mass}, up to a quadratic error in the short-to-long side ratio.

\begin{lemma}\label{lem:soft-coverage}
Suppose that $u,v\ne0$ and $s=|Tv|/|Tu|>\tau$. Then
\begin{equation}\label{eq:soft-coverage}
 \int_0^\infty
 \1_{\{|Tu|\le3R/\tau\}}\mathcal H_R(x,u,v)\,\frac{dR}{R}
 \ge\kappa_L-\frac{C_\phi L^2}{s^2}.
\end{equation}
\end{lemma}

\begin{proof}
For $h\in\R^2$, set
$\psi_h(Z)=\phi(Z-h)\phi(Z+h)$.
Then $\norm{\psi_h}_{L^1}\le1$ and
\begin{align*}
 \psi_h-\phi^2
 &=\bigl(\phi(\,\cdot-h)-\phi\bigr)
   \bigl(\phi(\,\cdot+h)-\phi\bigr)\\
 &\quad+\phi\bigl(\phi(\,\cdot-h)+\phi(\,\cdot+h)-2\phi\bigr).
\end{align*}
The $L^2$ translation bounds
\[
 \norm{\phi(\,\cdot\pm h)-\phi}_{L^2}
      \le |h|\norm{\nabla\phi}_{L^2},\qquad
 \norm{\phi(\,\cdot-h)+\phi(\,\cdot+h)-2\phi}_{L^2}
      \le |h|^2\norm{D^2\phi}_{L^2} 
\]
therefore imply
\begin{equation}\label{eq:second-order-density}
 \norm{\psi_h-\phi^2}_{L^1}\le C_\phi|h|^2.
\end{equation}
In addition, $\psi_h(Z)$ is supported in
the unit ball, by convexity.

In \eqref{eq:coverage-kernel}, make the change of variables
\[
 z=T(x+u/2)+rZ_1,\qquad
 w=T(x+v+u/2)+rZ_2.
\]
Put
\[
t=\frac{|Tv|}{R},
\qquad
n=\frac{Tv}{|Tv|},
\qquad
h(t)=\frac{Lt}{2|Tv|}Tu=\frac{Tu}{2r}.
\]
Then evenness of $\phi$ gives
\[
\Phi_z(x)\Phi_z(x+u)=\psi_{h(t)}(Z_1),\qquad
\Phi_w(x+v)\Phi_w(x+u+v)=\psi_{h(t)}(Z_2).
\]
Then $\frac{w-z}{R}=tn+(Z_2-Z_1)/L, d\gamma_r(z)d\gamma_r(w)=dZ_1dZ_2.$ Consequently,
\begin{align*}
	 \int_0^\infty\mathcal H_R(x,u,v)\,\frac{dR}{R} =
	\int_0^\infty\iint
	\1_{\{1\le|tn+(Z_2-Z_1)/L|<2\}}
	\psi_{h(t)}(Z_1)\psi_{h(t)}(Z_2)
	\,dZ_1dZ_2\,\frac{dt}{t}.
\end{align*}
Replacing both factors $\psi_{h(t)}$ by $\phi^2$ gives
$\kappa_L$, by Lemma~\ref{lem:exact-mass}. Since both densities are supported in the unit ball and the annulus condition requires
\(1\le \left|tn+(Z_2-Z_1)/L\right|<2\),
either integrand can be nonzero only when
\[
t\in I_L:=\left[1-\frac{2}{L},\,2+\frac{2}{L}\right].
\]
On this interval, $|h(t)|=Lt/(2s)\le2L/s$.
Using $\|\psi_{h(t)}\|_{L^1},\|\phi^2\|_{L^1}\le1$ and
\eqref{eq:second-order-density}, we obtain
\begin{align*}
	\left|
	\int_0^\infty\mathcal H_R(x,u,v)\,\frac{dR}{R}
	-\kappa_L
	\right|	\le2\int_{I_L}\|\psi_{h(t)}-\phi^2\|_{L^1}\,\frac{dt}{t}\le C_\phi\frac{L^2}{s^2}.
\end{align*}
Finally, on the support of $\mathcal H_R(x,u,v)$,
\[
\frac{|Tu|}{R}=\frac ts
\le\frac{2+2/L}{s}<\frac3\tau.
\] 
Thus the short-side cutoff leaves the integral unchanged,
proving \eqref{eq:soft-coverage}.
\end{proof}

We now sum the parallelogram energy with its coverage weight,
\begin{equation}\label{eq:localized-energy}
 Q_R(f)=
 \sum_{\substack{x,u,v\in\Z^2\\0<|Tu|\le3R/\tau}}
 f(x)f(x+u)f(x+v)f(x+u+v)
 K(B(u,v))\mathcal H_R(x,u,v).
\end{equation}
Equivalently, $Q_R$ integrates the mixed interaction with $f_z$
at $x,x+u$ and $f_w$ at $x+v,x+u+v$ over $\cA_R$.
The cutoff $|Tu|\le3R/\tau$ retains every covered thin
parallelogram by Lemma~\ref{lem:soft-coverage} and, since it depends only on $u$, preserves
the positive kernel in the base points. The next lemma bounds $Q_R$ by local self-energies.

\begin{lemma}\label{lem:localized-gram}
For every $R>0$,
\begin{equation}\label{eq:localized-gram}
 2Q_R(f)\le\iint_{\cA_R}
       \frac{\mu_w}{\mu_z}\bigl(\cE_\alpha(f_z)-\nu_z\bigr)
       \,d\gamma_r(z)d\gamma_r(w),
\end{equation}
where the integrand is defined to be zero if $\mu_z\mu_w=0$.
\end{lemma}

\begin{proof}
If $R<\lambda\tau/3$, the sum defining $Q_R$ is empty and the claim follows from
$\cE_\alpha(f_z)-\nu_z=\cE_\alpha^*(f_z)+2J(f_z)\ge0$ by
\eqref{eq:zero-edges}.
Assume from now on that $R\ge\lambda\tau/3$.

\emph{Step 1. Angular restriction.}
For $(z,w)\in\cA_R$, write $\omega=(w-z)/|w-z|$ and define
\[
 \Gamma_{z,w}
 =\left\{u\in\Z^2\setminus\{0\}:
        \left|\frac{Tu}{|Tu|}\cdot\omega\right|\le\frac4L\right\}.
\]
We state two facts on $\Gamma_{z,w}$. First, the short-side vector $u$ of every nonzero mixed term
lies in $\Gamma_{z,w}$. Indeed,
for a positive summand in the integrand defining $Q_R(f)$,
$Tx$ is within $r$ of $z$ and $T(x+v)$ is within $r$ of $w$.
Hence $|Tv-(w-z)|\le2r$, and $K(B(u,v))>0$ implies
\begin{equation}\label{eq:gram-angular-bound}
 \left|\frac{Tu}{|Tu|}\cdot\omega\right|
 \le\frac{\eta}{\lambda R}+\frac2L
 \le\frac{3\eta}{\lambda^2\tau}+\frac2L
 \le\frac4L.
\end{equation}

Second, two nonzero vectors \(u\) and \(h\) with \(K(B(u,h))>0\)
cannot both belong to \(\Gamma_{z,w}\).
To see this, put $\delta=4/L$ and resolve the
two unit vectors along $\omega$ and $\omega^\perp$.
Their $\omega$ components have absolute value at most $\delta$,
while their perpendicular components have absolute value at
least $\sqrt{1-\delta^2}$. It follows that
\[
 \left|\frac{Tu}{|Tu|}\cdot\frac{Th}{|Th|}\right|
 \ge(1-\delta^2)-\delta^2=1-\frac{32}{L^2}.
\]
On the other hand, positivity of $K(B(u,h))$ gives an upper
bound $\eta/\lambda^2\le10^{-6}$ for the same quantity.
These bounds are incompatible for $L\ge128$, see \eqref{eq:localization-parameters}.

\emph{Step 2. Positive kernel estimate.}
For each fixed $(z,w)\in\cA_R$, set
\[
 \chi_{z,w}(u)=
 \1_{\{0<|Tu|\le3R/\tau\}}\1_{\Gamma_{z,w}}(u),
 \qquad
 F_{z,u}(x)=f_z(x)f_z(x+u),
\]
\[
 G_z(u)=\sum_{x,y}F_{z,u}(x)F_{z,u}(y)K(B(u,y-x)),
 \qquad
 S_{z,w}=\sum_{u\ne0}\chi_{z,w}(u)G_z(u).
\]

For $\mu_z\mu_w>0$, the kernel representation
\eqref{eq:positive-kernel-representation} and the weighted
arithmetic--geometric mean inequality give
\begin{equation}\label{eq:gram-mixed-bound}
 \sum_{x,y}F_{z,u}(x)F_{w,u}(y)K(B(u,y-x))
 \le\sqrt{G_z(u)G_w(u)}
 \le\frac12\left(\frac{\mu_w}{\mu_z}G_z(u)
                  +\frac{\mu_z}{\mu_w}G_w(u)\right).
\end{equation}
By the angular restriction \eqref{eq:gram-angular-bound},
we may insert $\chi_{z,w}(u)$ into the sum defining $Q_R(f)$.
Changing variables from $v$ to $y=x+v$ in that sum and
then applying \eqref{eq:gram-mixed-bound}, we obtain
\begin{align}
 Q_R(f)
 &=\iint_{\cA_R}\sum_{u\ne0}\chi_{z,w}(u)
   \sum_{x,y}F_{z,u}(x)F_{w,u}(y)K(B(u,y-x))
   \,d\gamma_r(z)d\gamma_r(w)\notag\\
 &\le\frac12
   \iint_{\substack{(z,w)\in\cA_R\\\mu_z\mu_w>0}}
   \left(\frac{\mu_w}{\mu_z}S_{z,w}
        +\frac{\mu_z}{\mu_w}S_{w,z}\right)
   \,d\gamma_r(z)d\gamma_r(w)\notag\\
 &=\iint_{\substack{(z,w)\in\cA_R\\\mu_z\mu_w>0}}
   \frac{\mu_w}{\mu_z}S_{z,w}
   \,d\gamma_r(z)d\gamma_r(w).
 \label{eq:gram-after-cauchy}
\end{align}
The last equality uses $\chi_{z,w}=\chi_{w,z}$ and symmetry under $z\leftrightarrow w$.

\emph{Step 3. Exchange of the two side vectors.}
The angular exclusion above (see Step~1) gives
\begin{equation}\label{eq:cone-budget}
 \chi_{z,w}(u)+\chi_{z,w}(h)\le1
 \quad\text{if }u,h\ne0\text{ and }K(B(u,h))>0.
\end{equation}
With $h=y-x$, the identity
\[
 F_{z,u}(x)F_{z,u}(x+h)
 =F_{z,h}(x)F_{z,h}(x+u)
\]
allows us to average over $u$ and $h$ when both are nonzero.
Separating $h=0$ and using $K(0)=1$, we obtain
\begin{align}
 S_{z,w}
 &=\frac12\sum_{\substack{x,u,h\\u,h\ne0}}
   \bigl(\chi_{z,w}(u)+\chi_{z,w}(h)\bigr)
   F_{z,u}(x)F_{z,u}(x+h)K(B(u,h))\notag\\
 &\quad+\sum_{x,u\ne0}
   \chi_{z,w}(u)f_z(x)^2f_z(x+u)^2\notag\\
 &\le\frac12\cE_\alpha^*(f_z)+J(f_z)
 =\frac12\bigl(\cE_\alpha(f_z)-\nu_z\bigr).
 \label{eq:gram-local-bound}
\end{align}
The inequality follows from \eqref{eq:cone-budget} and
$0\le\chi_{z,w}\le1$.
Substituting \eqref{eq:gram-local-bound} into
\eqref{eq:gram-after-cauchy} proves
\eqref{eq:localized-gram}.
\end{proof}

We now combine these estimates with Corollary~\ref{cor:aspect-sums}.

\begin{proof}[Proof of Lemma~\ref{lem:energy-recursion}]
Put $m=\log\#S$ and define
\[
 W_f(x,u,v)
:=f(x)f(x+u)f(x+v)f(x+u+v)K(B(u,v)).
\]
Set
\begin{align*}
 \cE_{\rm mod}(f)
:=\sum_{\substack{x,u,v\in\Z^2\\u,v\ne0\\
\sigma(u,v)\le\tau}}W_f(x,u,v), \qquad
 \cE_{\rm thin}(f)
:=\sum_{\substack{x,u,v\in\Z^2\\u,v\ne0\\
\sigma(u,v)>\tau}}W_f(x,u,v).
\end{align*}
Then $\cE_\alpha^*(f)=\cE_{\rm mod}(f)+\cE_{\rm thin}(f)$,
and Corollary~\ref{cor:aspect-sums} gives
\begin{align}\label{eq:recur-E-mod-bound}
\cE_{\rm mod}(f)\le C_I\tau mJ(f).    
\end{align}

Since $\tau>1$ and $W_f(x,u,v)=W_f(x,v,u)$,
\begin{align}\label{eq:recur-E-thin-rewrite}
 \cE_{\rm thin}(f)
=2\sum_{\substack{x,u,v\in\Z^2\\u,v\ne0\\
|Tv|>\tau|Tu|}}W_f(x,u,v).
\end{align}
For $u,v\ne0$ with $s=|Tv|/|Tu|>\tau$,
Lemma~\ref{lem:soft-coverage} yields
\begin{align}\label{eq:recur-coerage}
 \kappa_L W_f(x,u,v)
 \le\int_0^\infty
   \1_{\{|Tu|\le3R/\tau\}}W_f(x,u,v)\mathcal H_R(x,u,v)
   \,\frac{dR}{R}
   +\frac{C_\phi L^2}{s^2}W_f(x,u,v).
\end{align}
At each scale $R$, \eqref{eq:localized-energy} implies
\begin{align}\label{eq:recur-scale}
\sum_{\substack{x,u,v\in\Z^2\\u,v\ne0\\
|Tv|>\tau|Tu|}}
 \1_{\{|Tu|\le3R/\tau\}}
 W_f(x,u,v)\mathcal H_R(x,u,v)
 \le Q_R(f).
\end{align}
Summing the coverage estimate \eqref{eq:recur-coerage}, using
\eqref{eq:recur-E-thin-rewrite} and \eqref{eq:recur-scale}, 
we obtain
\begin{align}\label{eq:E-thin-1}
 \kappa_L\cE_{\rm thin}(f)
 \le2\int_0^\infty Q_R(f)\,\frac{dR}{R}+C_\phi L^2
\sum_{\substack{x,u,v\in\Z^2\\u,v\ne0\\\sigma(u,v)>\tau}}\frac{W_f(x,u,v)}{\sigma(u,v)^2}.
\end{align}
By \eqref{eq:aspect-error}, the last term in \eqref{eq:E-thin-1} is bounded by
$C_{I,\phi}L^2\tau^{-1}mJ(f)$.

Integrating Lemma~\ref{lem:localized-gram} and using
\eqref{eq:entropy-localization}, we have
\begin{align}\label{eq:E-thin-2}
\frac2{\kappa_L}\int_0^\infty Q_R(f)\,\frac{dR}{R}
&\le\frac1{\kappa_L}\int_0^\infty\iint_{\cA_R}
       \frac{\mu_w}{\mu_z}
    \bigl(\cE_\alpha(f_z)-\nu_z\bigr)
\,d\gamma_r(z)d\gamma_r(w)\frac{dR}{R}\nonumber\\
 &=\bigl(\mathscr L(\cE_\alpha-\nu)\bigr)(f).
\end{align}
Dividing \eqref{eq:E-thin-1} by $\kappa_L$,
applying the bound \eqref{eq:E-thin-2}, and adding the estimate for
$\cE_{\rm mod}(f)$ gives, since $\kappa_L\ge\log2$,
\[
 \cE_\alpha^*(f)
 \le\bigl(\mathscr L(\cE_\alpha-\nu)\bigr)(f)
+C_{I,\phi}\left(\tau+\frac{L^2}{\tau}\right)mJ(f).
\]
The parameters $\phi,L,\tau$ are fixed, so the error is bounded
by $C_ImJ(f)$. 

Using
$\cE_\alpha-\nu=G+J$, $G=\cE_\alpha^*+J$, and
\eqref{eq:operator-pair-mass}, we conclude that
\begin{align*}
 G(f)
 &\le(\mathscr LG)(f)+(\mathscr LJ)(f)+J(f)+C_ImJ(f)\\
 &\le(\mathscr LG)(f)+(2+C_Im)J(f).
\end{align*}
Since $m\ge\log3$, enlarging $C_I$ proves
\eqref{eq:entropy-energy-recursion}.
\end{proof}

\end{document}